\documentclass[11pt]{article}
\usepackage[utf8]{inputenc}
\usepackage{epsfig}

\usepackage{graphicx}
\usepackage[usenames,dvipsnames]{color}

\newcommand{\nat}{\ensuremath {\mathbb N} }

\newcommand{\Prob}{\mathbb{P}}
\newcommand{\eps}{\varepsilon}
\newcommand{\E}{\mathbb E}

\newcommand{\aas}{{\sl{a.a.s.}}}

\newcommand{\bin}{\mathrm{Bin}}

\usepackage[numbers]{natbib}
\setcitestyle{aysep={}}

\usepackage{graphicx}
\usepackage[colorlinks=true, allcolors=blue]{hyperref}

\usepackage{paralist}
\usepackage{amsmath}
\usepackage{amssymb}
\usepackage{amsthm}
\usepackage{algorithmic}
\usepackage{algorithm}
\usepackage{tikz}

\newtheorem{theorem}{Theorem}[section]
\newtheorem{lemma}[theorem]{Lemma}
\newtheorem{cor}[theorem]{Corollary}

\author{
Lenar Iskhakov\thanks{Advanced Combinatorics and Network Applications Lab, Moscow Institute of Physics and Technology, Moscow, Russia} 
\and
Bogumi\l{} Kami\'nski\thanks{SGH Warsaw School of Economics, Warsaw, Poland} 
\and 
Maksim Mironov\footnotemark[1] 
\and
Pawe\l{}~Pra\l{}at\thanks{Department of Mathematics, Ryerson University, Toronto, ON, Canada}\thanks{Corresponding author} 
\and 
Liudmila Prokhorenkova\footnotemark[1] \thanks{Machine intelligence and research department, Yandex, Moscow, Russia} 
}

\title{Local Clustering Coefficient of Spatial Preferential Attachment Model}
\date{}

\begin{document}

\maketitle

\begin{abstract}
In this paper, we study the clustering properties of the Spatial Preferential Attachment (SPA) model. This model naturally combines geometry and preferential attachment using the notion of spheres of influence. It was previously shown in several research papers that graphs generated by the SPA model are similar to real-world networks in many aspects. Also, this model was successfully used for several practical applications. However, the clustering properties of the SPA model were not fully analyzed. The clustering coefficient is an important characteristic of complex networks which is tightly connected with its community structure. In the current paper, we study the behaviour of $C(d)$, which is the average local clustering coefficient for the vertices of degree $d$. It was empirically shown that in real-world networks $C(d)$ usually decreases as $d^{-a}$ for some $a>0$ and it was often observed that $a=1$. We prove that in the SPA model $C(d)$ decreases as $1/d$. Furthermore, we are also able to prove that not only the average but the individual local clustering coefficient of a vertex $v$ of degree $d$ behaves as $1/d$ if $d$ is large enough. The obtained results further confirm the suitability of the SPA model for fitting various real-world complex networks.
\end{abstract}


\section{Introduction}

The evolution of complex networks has attracted a lot of attention in recent years. Empirical studies of different real-world networks have shown that such networks have some typical properties: small diameter, power-law degree distribution, clustering structure, and others~\cite{costa2007characterization,newman2003structure,thadakamalla2015}.
Therefore, numerous random graph models have been proposed to reflect and predict such quantitative and topological aspects of growing real-world networks~\cite{boccaletti2006complex,bollobas2003mathematical}.

The most well studied property of complex networks is their vertex degree distribution. For the majority of studied real-world networks, the degree distribution follows a power law with a parameter $\gamma$ which usually belongs to $(2,3)$~ \cite{barabasi1999emergence,faloutsos1999power,newman2005power}.

Another important property of real-world networks is their clustering (or community) structure. The presence of such a structure highly affects many processes occurring in networks like promotion of products via viral marketing or the spreading of computer viruses and infection diseases~\cite{prokhorenkova2019learning}. One way to characterize the presence of clustering structure is to measure the {\it clustering coefficient}, which is, roughly speaking, the probability that two neighbours of a vertex are connected. There are two well-known formal definitions: the global clustering coefficient and the average local clustering coefficient (see Section~\ref{sec:clustering} for definitions). At some point, it was believed that for many real-world networks both the average local and the global clustering coefficients tend to non-zero limit as the network becomes large; for example, some numerical values can be found in~\cite{newman2003structure}; however, this statement for the global clustering coefficient is questionable and recently some contradicting theoretical results were presented in~\cite{ostroumova2016global}.

In this paper, we mostly focus on the behaviour of $C(d)$, which is the average local clustering coefficient for the vertices of degree $d$. The function $C(d)$ gives a better insight into the network structure than just the average clustering coefficient. It was empirically shown that in real-world networks $C(d)$ usually decreases as $d^{-\psi}$ for some $\psi>0$~\cite{csanyi2004structure, leskovec2008dynamics, serrano2006clustering1, vazquez2002large}. In particular, for many studied networks, $C(d)$ scales as $d^{-1}$~\cite{ravasz2003hierarchical}. Moreover, it was shown in~\cite{serrano2006clustering1,serrano2006clustering2} that the behaviour of $C(d)$ is tightly connected to the notion of weak and strong transitivity, which, in turn, affects percolation properties of a network.

We study the clustering properties of the \emph{Spatial Preferential Attachment} (SPA) model introduced in~\cite{spa1}. This model combines geometry and preferential attachment; 
the formal definition is given in Section~\ref{sec:spa_def}. It was previously shown that graphs generated by the SPA model are similar to real-world networks in many aspects.  For example, it was proven in \cite{spa1} that the vertex degree distribution follows a power law. Also, this model was successfully used for several practical applications like the analysis of a duopoly market~\cite{SPAKaminski2017}. More details on the properties and applications of the SPA model are given in Section~\ref{sec:spa_prop}. However, the clustering coefficient $C(d)$, which is an extremely important characteristic, was not previously analyzed for this model. Some basic clustering properties of the closely related model were analyzed: it is proved in \cite{jacob2013spatial} and \cite{jacob2015spatial} that the average local clustering coefficient converges in probability to a strictly positive limit; also, the global clustering coefficient converges to a nonnegative limit, which is nonzero if and only if the power-law degree distribution has a finite variance.

A short version of this article was published as a conference proceeding~\cite{iskhakov2018clustering} and was mostly focused on the behaviour of the clustering coefficient on simulated graphs. In the current paper we provide a thorough theoretical analysis of the asymptotic properties of the SPA model. However, we also include some empirical results from \cite{iskhakov2018clustering} to illustrate the obtained theoretical counterparts.

The rest of the paper is organized as follows. In the next section we define the SPA model and  discuss its properties. Then, in Section~\ref{sec:clustering}, we define the local clustering coefficient. The obtained theoretical results are discussed in Section~\ref{sec:results} and illustrated on simulated graphs in Section~\ref{sec:simulations}. All proofs are given in Section~\ref{sec:proofs}. Section~\ref{sec:conclusion} concludes the paper.

\section{Spatial Preferential Attachment model}\label{sec:spa}

\subsection{Definition}\label{sec:spa_def}

This paper focuses on the \emph{Spatial Preferential Attachment} (SPA) model, which was first introduced in~\cite{spa1}. This model combines preferential attachment with geometry by introducing ``spheres of influence'' whose volume grows with the degree of a vertex. The parameters of the model are the \emph{link probability} $p\in[0,1]$ and two constants $A_1,A_2$ such that $0 < A_1 < \frac{1}{p}$, $A_2>0$. All vertices are placed in the $m$-dimensional unit hypercube $S = [0,1]^m$ equipped with the torus metric derived from any of the $L_k$ norms, i.e.,  
\begin{equation}
d(x,y)=\min \big\{ ||x-y+u||_k\,:\,u\in \{-1,0,1\}^m \big\} \,\,\,\,\,\,\, \forall  x,y \in S \,.
\end{equation}
The SPA model generates a sequences random directed graphs $\{G_{t}\}$, where $G_{t}=(V_{t},E_{t})$, $V_{t}\subseteq S$. Let $\deg^{-}(v,t)$ be the in-degree of the vertex $v$ in $ G_{t}$, and $\deg^+(v,t)$ its out-degree. Then, the \emph{sphere of influence} $S(v,t)$ of the vertex $v$ at time $t\geq 1$ is the ball centered at $v$ with the following volume:
\begin{equation}
|S(v,t)|=\min\left\{\frac{A_1{\deg}^{-}(v,t)+A_2}{t},1\right\}.
\end{equation}

In order to construct a sequence of graphs we start at $t=0$ with $G_0$ being the null graph. At each time step $t$ we construct $G_{t}$ from $G_{t-1}$ by, first, choosing a new vertex $v_t$ \emph{uniformly at random} from $S$ and adding it to $V_{t-1}$ to create $V_{t}$. Then, independently, for each vertex $u\in V_{t-1}$ such that $v_t \in S(u,t-1)$, a directed link $(v_{t},u)$ is created with probability $p$. Thus, the probability that a link $(v_t,u)$ is added in time-step $t$ equals $p\,|S(u,t-1)|$.
See Figure~\ref{fig1} for a drawing of a simulation of the SPA model~\cite{spa1}.

\begin{figure} [ht] \label{fig1}
\begin{center}
\includegraphics[width=0.6\textwidth]{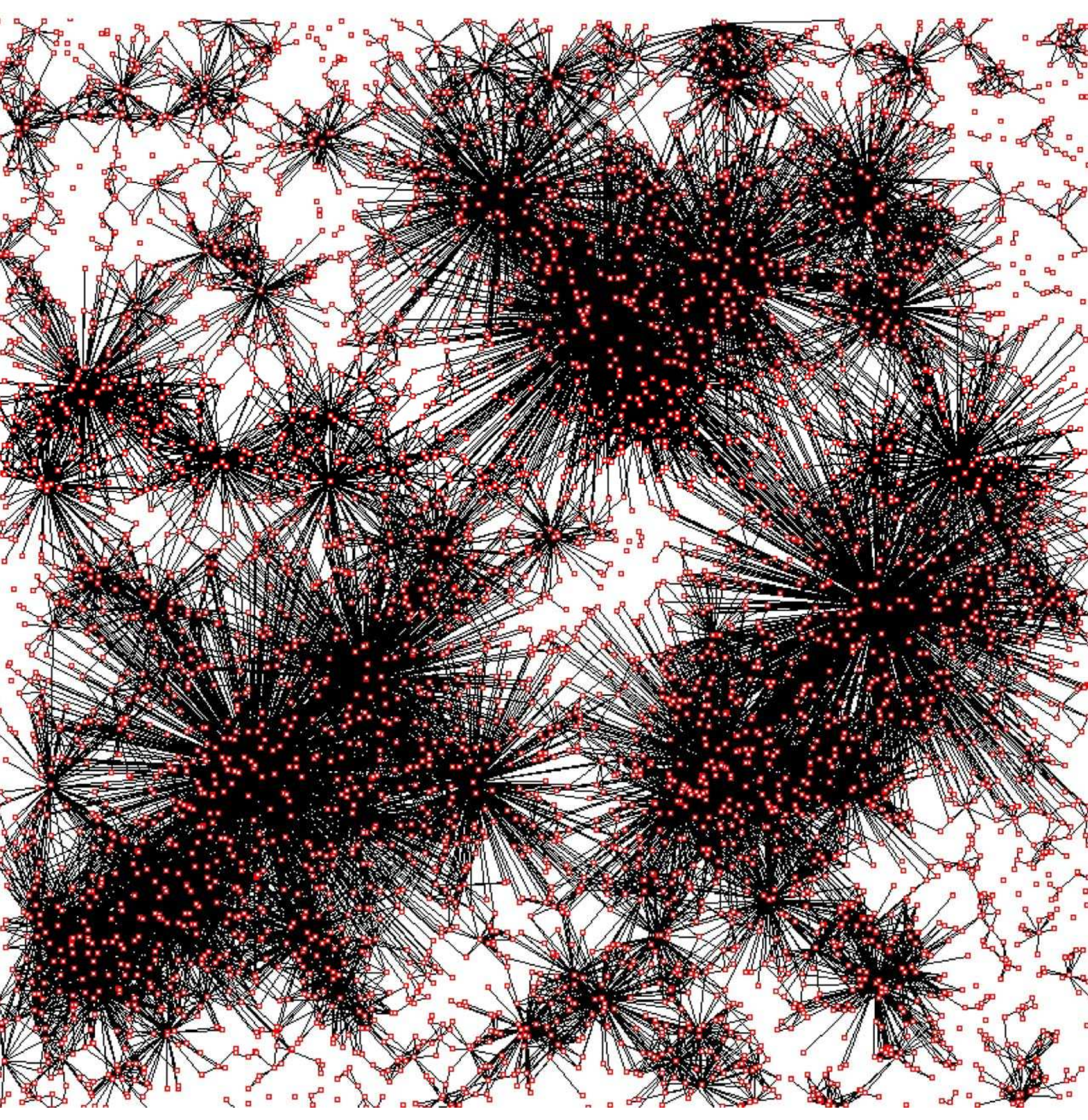}\caption{A simulation on the unit square with $t=5000$, $p=1$, and $A_1=A_2=1.$}
\end{center}
\end{figure}

\subsection{Properties and applications of the model}\label{sec:spa_prop}

In this section, we briefly discuss previous studies on properties and applications of the SPA model. This model is known to produce scale-free networks, which exhibit many of the characteristics of real-life networks \cite[see][]{spa1,spa2}. 
Specifically,~\cite{spa1} (Theorem~1.1) proved that the SPA model generates graphs with a power law in-degree distribution with coefficient $1 + 1/(pA_1)$. On the other hand, the average out-degree is asymptotic to $pA_2/(1-pA_1)$ \cite[see Theorem~1.3 in][]{spa1}.
In~\cite{spa4}, some properties of common neighbours were used to explore the underlying geometry of the SPA model and quantify vertex similarity based on the distance in the space.
Usually, the distribution of vertices in $S$ is assumed to be uniform~\cite{spa4}, but~\cite{spa5} also investigated non-uniform distributions, which is clearly a more realistic setting. 

Let us briefly discuss the parameters of the model: $p, A_1$, and $A_2$. The parameter $p$ is usually highly influenced by the application; for example, if one wants to model the citation network, $p$ would correspond to the ratio of the average number of papers cited and the number of papers a typical author is aware of (presumably larger for, say, computer science papers; lower for mathematics, etc.). Then, $A_1$ controls the degree distribution and $A_2$ can be used to tune the average degree.

Importantly, in~\cite{spa3}, it was shown that the SPA model gave the best fit, in terms of graph structure, for a series of social networks derived from Facebook. This means that this model is a good synthetic approximation for some real-world networks and can be suitable for various practical applications. For instance, the SPA model was used to analyze a duopoly market on which there is uncertainty of a product quality~\cite{SPAKaminski2017} and to model the interpersonal network of top managers~\cite{morgan2018cognition}. Also, \cite{feldman2017high} used the SPA model to study the interaction between community structure and the spread of infections in complex networks. Fitting the SPA model to real-world networks can also potentially be used for link prediction problems. The model is especially useful when some underlying structure that affects the network is not easily measurable but is important for the application. Consider, for example, a citation network in which vertices correspond to papers and directed edges correspond to citations between them. Clearly, the content of the paper (that is, location of the corresponding vertex in $S$) affects edges between vertices but very often the content is not provided. Assuming that the citation network is similar to the SPA model (and after a careful tuning of parameters), one can use the model and results from~\cite{spa4,spa5} to predict the similarity between papers and then use some clustering algorithm to extract papers on a similar topic. The same methodology applies to a more sophisticated scenarios such as predicting people's taste/hobbies/believes based on social networks they are part of. 

As we already mentioned, one of the most important properties of real-world networks is their clustering structure: it highly affects many processes occurring in networks. Therefore, to understand the suitability of the SPA model for various applications, it is crucial to analyze its clustering structure. The first step in this direction was made in~\cite{prokhorenkova2017modularity,modularity}, where modularity of the SPA model was investigated, which is a global criterion to define communities and a way to measure the presence of community structure in a network. In the current paper we analyze the clustering structure using another characteristic~--- the average local clustering coefficient $C(d)$.

\section{Clustering coefficient}\label{sec:clustering}

Clustering coefficient measures how likely two neighbours of a vertex are connected by an edge.
There are several definitions of clustering coefficient proposed in the literature (see, e.g.,~\cite{bollobas2003mathematical}). 

The {\it global clustering coefficient} $C_{glob}(G)$ of a graph $G$ is the ratio of three times the number of triangles to the number of pairs of adjacent edges in $G$. 
In other worlds, if we sample a random pair of adjacent vertices in $G$, then $C_{glob}(G)$ is the probability that these three vertices form a triangle. The global clustering coefficient in the SPA model was previously studied in~\cite{jacob2013spatial, jacob2015spatial} and it was proven that $C_{glob}(G_n)$ converges to a  limit, which is positive if and only if the power-law degree distribution has a finite variance.

In this paper, we focus on the {\it local clustering coefficient}. 
Let us first define it for an undirected graph $G = (V, E)$.
Let $N(v)$ be the set of neighbours of a vertex $v$, $|N(v)| = \deg(v)$. For any $B \subseteq V$, let $E(B)$ be the set of edges in the graph induced by the vertex set $B$; that is,
\begin{equation}
E(B) = \{ \{ u,w\}  \in E : u, w \in B \}.  
\end{equation}
Finally, \emph{clustering coefficient} of a vertex $v$ is defined as follows:
\begin{equation}
c(v) = |E(N(v))| \Big/ \binom{\deg(v)}{2}.
\end{equation}
Clearly, $0 \le c(v) \le 1$.

Note that the local clustering $c(v)$ is defined individually for each vertex and it can be noisy, especially for the vertices of not too large degrees. Therefore, the following characteristic was extensively studied in the literature. Let $C(d)$ be the local clustering coefficient averaged over the vertices of degree $d$; that is,
\begin{equation}
C(d) = \frac{\sum_{v: \deg(v) = d}c(v)}{|\{v: \deg(v) = d\}|}\,.
\end{equation}
Further in the paper we will also use the notation $c(v,t)$ and $C(d,t)$ referring to graphs on $t$ vertices.

The local clustering $C(d)$ was extensively studied both theoretically and empirically. 
For example, it was observed in a series of papers that in real-world networks $C(d) \propto d^{-\varphi}$ for some $\varphi > 0$.
In particular, \cite{ravasz2003hierarchical} shows that $C(d)$ can be well approximated by $d^{-1}$ for four large networks, \cite{vazquez2002large} obtains power-law in a real network with parameter 0.75, while \cite{csanyi2004structure} obtain $\varphi = 0.33$.
The local clustering coefficient was also studied in several random graph models of complex networks. For instance, it was shown in \cite{dorogovtsev2002pseudofractal, krot2017local, newman2003properties} that some models have $C(d) \propto d^{-1}$. As we prove in this paper, similar behaviour is observed in the SPA model.

Recall that the graph $G_t$ constructed according to the SPA model is directed. Therefore, we first analyze the directed version of the local clustering coefficient and then, as a corollary, we obtain the corresponding results for the undirected version. Let us now define the directed clustering coefficient. By $N^-(v,t) \subseteq V_t$ we denote the set of in-neighbours of a vertex $v$ at time $t$; $\deg^-(v,t) = |N^-(v,t)|$. So, the directed clustering coefficient of vertex $v$ at time $t$ is defined as
\begin{equation}
c^{-}(v,t) = |E(N^-(v,t))| \Big/ \binom{\deg^-(v,t)}{2},
\end{equation}
where this time $E(B) = \{ (u,w) \in E : u, w \in B \}$ for any $B \subseteq V_t$.
Similarly to the undirected case, we define
\begin{equation}
C^{-}(d,t) = \frac{\sum_{v: \deg^-(v,t) = d}c^-(v,t)}{|\{v: \deg^-(v,t) = d\}|}\,.
\end{equation}

\section{Results}\label{sec:results}

Let us start with introducing some notation. As typical in random graph theory, all results in this paper are asymptotic in nature; that is, we aim to investigate properties of $G_n$ for $n$ tending to infinity. As a result, we will be always allowed to assume that $n$ is large enough for some inequalities/properties to hold. We say that an event holds \emph{asymptotically almost surely} (\aas) if it holds with probability tending to one as $n\to\infty$. Also, given a set $S$ we say that \emph{almost all} elements of $S$ have some property $P$ if the number of elements of $S$ that do not have $P$ is $o(|S|)$ as $n \to \infty$. Throughout the paper, the standard notations $o(\cdot)$, $O(\cdot)$, $\Theta(\cdot)$, $\Omega(\cdot)$ refer to functions whose growth is bounded as $n \to \infty$. We use the notation $f \ll g$ for $f=o(g)$ and $f \gg g$ for $g=o(f)$. We also write $f \sim g$ if $f/g \to 1$ as $n \to \infty$ (that is, when $f = (1+o(1)) g$). 
Let us illustrate how we combine this notation in probabilistic statements. For example, we say that $f \sim g$ with probability $1-o(p(n))$ if there exists a function $\varepsilon=\varepsilon(n)=o(1)$ such that $g (1-\varepsilon) < f < g(1+\varepsilon)$ with probability at least $1-\varepsilon p(n)$ (of course, provided that $n$ is large enough). 
Finally, let us emphasize that $f$ and $g$ are functions of $n$ and possibly some other variables, but all asymptotic relations presented in this paper are for the number of vertices $n$ tending to infinity.


Let us first consider the directed clustering coefficient. It turns out that for the SPA model we are able not only to prove the asymptotics for $C^-(d,n)$, which is the average clustering over all vertices of in-degree $d$, but also analyze the individual clustering coefficients $c^-(v,n)$, which is a much stronger result. However, in order to do this, we need to assume that $\deg^-(v,n)$ is large enough.

From technical point of view, it will be convenient to partition the set of contributing edges, $E(N^-(v,n))$, and independently consider edges to ``old'' and to ``young'' neighbours of $v$. Formally, let us take any function $\omega(n)$ that tends to infinity (arbitrarily slowly) as $n\to \infty$; for example, $\omega(n) = \log \log \log n$ or even $\omega(n)$ could be the inverse Ackermann function that is less than 5 for any practical input size $n$. The function $\omega(n)$ will remain fixed throughout the rest of the paper. Let $\hat{T}_v$ be the smallest integer $t$ such that $\deg^-(v,t)$ exceeds $\omega \log n$ (or $\hat{T}_v=n$ if $\deg^-(v,n) < \omega \log n$). The reason for introducing $\hat{T}_v$ is that the behaviour of vertices is chaotic and unpredictable at first but it stabilizes once they accumulate enough neighbours; the threshold happens to be around $\log n$. Vertices in $N^-(v,\hat{T}_v)$ are called \emph{old neighbours of $v$}; $N^-(v,n) \setminus N^-(v,\hat{T}_v)$ are \emph{new neighbours of $v$}. So, we can partition $E(N^-(v,n))$ into $E_{old}(N^-(v,n))$ and $E_{new}(N^-(v,n))$, the first contains the edges going from neighbours of $v$ to its old neighbours, the second contains the remaining edges, i.e., ones connecting only new neighbours. Again, the reason for partitioning $E(N^-(v,n))$ is that old neighbours of $v$ are unpredictable (but, fortunately, there are few of them); on the other hand, the behaviour of young neighbours can be well understood. Formally,
$$
E_{old}(N^-(v,n)) = \{ (u,w) \in E_n : u \in N^-(v,n), w \in N^-(v,\hat{T}_v) \},
$$
$$
E_{new}(N^-(v,n)) = E(N^-(v,n)) \setminus E_{old}(N^-(v,n))\,;
$$
and
\begin{equation}
c^-(v,n) = c_{old}(v,n) + c_{new}(v,n),\label{eq:old_new}
\end{equation}
where
\begin{eqnarray*}
c_{old}(v,n) &=& |E_{old}(N^-(v,n))| \Big/ \binom{\deg^-(v,n)}{2}, \\
c_{new}(v,n) &=& |E_{new}(N^-(v,n))| \Big/ \binom{\deg^-(v,n)}{2}.
\end{eqnarray*}

\bigskip

Let us start with the following theorem which is extensively used in our reasonings and is interesting and important on its own. Variants of this results were proved in~\cite{spa4, spa5}; here, we present a slightly modified statement from~\cite{spa5}, adjusted to our current needs. We provide the proof in Section~\ref{sec:proof-degconc} for completeness. 

\begin{theorem}\label{degconc}
Let $\omega=\omega(n)$ be any function tending to infinity together with $n$. The following holds with probability $1-o(n^{-4})$. For any vertex $v$ with
$$
\deg^-(v,n)=k=k(n) \geq \omega \log n
$$
and for all values of $t$ such that 
\begin{equation}
n \left(\frac{\omega \log n}{k}\right)^{\frac{1}{p A_1}} =: T_v \le t \le n,
\end{equation}
we have
\begin{equation}
\deg^-( v,t) \sim k \left(\frac{t}{n}\right)^{p A_1}.
\end{equation}
\end{theorem}

The expression for $T_v=T_v(n)$ is chosen so that at time $T_v$ vertex $v$ has $(1+o(1)) \omega \log n$ neighbours \aas\ The implication of this theorem is that once a vertex accumulates $\omega \log n$ neighbours, its behaviour can be predicted with high probability until the end of the process (that is, till time $n$) when its degree reaches $k$. In other words the following property holds. For a fixed value of $n$, let $M=M(n)$ and $m=m(n)$ be the maximum and, respectively, the minimum ratio between $\deg^-(v,t)$ and the deterministic function $k(t/n)^{pA_1}$ (taken over the interval $T_v \le t \le n$). Then, \aas\ both $M$ and $m$ tend to one as $n \to \infty$.

\bigskip

This property can be used to show that the contribution to $c^-(v,n)$ coming from edges to new neighbours of $v$ is well concentrated.

\begin{theorem}\label{thm:c_new}
Let $\omega=\omega(n)$ be any function tending to infinity together with $n$. Then, with probability $1-o(n^{-1})$ for any vertex $v$ with 
$$
\deg^-(v,n)=k=k(n) \geq (\omega \log n)^{4 + (4pA_1+2)/(pA_1(1-pA_1))}
$$
we have
$$
c_{new}(v,n) = \Theta(1/k).
$$
\end{theorem}

\bigskip

Unfortunately, if a vertex $v$ lands in a densely populated region of $S$, it might happen that $c_{old}(v,b)$ is much larger than $1/k$. We show the following `negative' result (without trying to aim for the strongest statement) that shows that there is no hope for extending Theorem~\ref{thm:c_new} to $c^-(v,n)$.

\begin{theorem}\label{thm:negative}
Let $C = 5 \log \left( 1/p \right)$ and 
$$
\xi = \xi(n) = 1 / (\omega (\log \log n)^2 (\log \log \log n)) = o(1)
$$ 
for some $\omega = \omega(n)$ tending to infinity as $n \to \infty$. Suppose that $k = k(n)$ is such that $2 \le k \le n^{\xi}.$
Then, a.a.s., there exists a vertex $v$ such that $\deg^-(v,n) \sim k$ and  
\begin{itemize}
\item [(i)] $c^-(v,n) = 1$, provided that $2 \le k \le \sqrt{\log n / C}$,
\item [(ii)] $c^-(v,n) = \Omega(1) \gg 1/k$, provided that  $\sqrt{\log n / C} \le k \le \log n / \log \log n$,
\item [(iii)] $c^-(v,n) \gg (\log \log n)^2 (\log \log \log n) / k \gg 1/k$, provided that \\ $\log n / \log \log n \le k \le n^{\xi}$.
\end{itemize}
\end{theorem}

On the other hand, Theorem~\ref{thm:c_new} implies immediately the following corollary.

\begin{cor}\label{cor:main}
Let $\omega=\omega(n)$ be any function tending to infinity together with $n$. The following holds with probability $1-o(n^{-1})$. For any vertex $v$ for which 
$$
\deg^-(v,n)=k=k(n) \geq (\omega \log n)^{4 + (4pA_1+2)/(pA_1(1-pA_1))}
$$
it holds that 
\begin{eqnarray*}
c^-(v,n) &\ge& c_{new}(v,n) = \Omega(1/k) \\
c^-(v,n) &=& c_{old}(v,n) + c_{new}(v,n) \\ &=& O(\omega \log n / k) + O(1/k) = O(\omega \log n / k).
\end{eqnarray*}
\end{cor}

This corollary states that for \textit{all} vertices with large enough degree $k$ we have the desired lower bound $\Omega(1/k)$ for the clustering coefficient $c^{-}(v,n)$. On the other hand, the upper bound, which is $O(\omega \log n / k)$, grows faster than desired due to the presence of $c_{old}(v,n)$. However, despite the `negative' result (stated in Theorem~\ref{thm:negative}), almost all vertices (of large enough degrees) have clustering coefficients of order $1/k$. Below is a precise statement. The conclusions in cases~(i)' and~(ii)' follow immediately from Theorem~\ref{thm:c_new}.

\begin{theorem}\label{thm:average}
Let $\eps, \delta \in (0,1/2)$ be any two constants, and let $k = k(n) \le n^{pA_1 - \eps}$ be any function of $n$. Let $X_k$ be the set of vertices of $G_n$ of in-degree between $(1-\delta)k$ and $(1+\delta)k$. 
Then, a.a.s., the following holds.
\begin{itemize}
\item [(i)] Almost all vertices in $X_k$ have $c_{old}(v,n) = O(1/k)$, provided that $k \gg \log^{C_1} n$, where $C_1 = (1+(2+\eps)pA_1)/(1-pA_1)$.
\item [(i)'] As a result, almost all vertices in $X_k$ have $c^-(v,n) = \Theta(1/k)$, provided that $k \gg \log^{C} n$, where $C = 4 + (4pA_1+2)/(pA_1(1-pA_1))$.
\item [(ii)] The average value of clustering coefficients $c_{old}(v,n)$ over the vertices in $X_k$ is 
$$
C_{old}(X_k) := \frac {1}{|X_k|} \sum_{v \in X_k} c_{old}(v,n) = O(1/k),
$$
provided that $k \gg \log^{C_2} n$, where $C_2 = (1+(2+pA_1+\eps)pA_1)/(1-~pA_1)$.
\item [(ii)'] As a result, the average value of clustering coefficients $c^-(v,n)$ over the vertices in $X_k$ is
$$
C^-(X_k) := \frac {1}{|X_k|} \sum_{v \in X_k} c^-(v,n) = \Theta(1/k),
$$
provided that $k \gg \log^{C} n$, where $C = 4 + (4pA_1+2)/(pA_1(1-pA_1))$.
\end{itemize}
\end{theorem}

\bigskip

Finally, let us discuss the undirected case. The following corollary holds.
\begin{cor}\label{cor:undir}
Let $c(v,n)$ be the clustering coefficient defined for the undirected graph $\hat G_n$ obtained from $G_n$ by considering all edges as undirected. Then
Corollary~\ref{cor:main} and
Theorem~\ref{thm:average} hold with replacing $c^-(v,n)$ by $c(v,n)$.
\end{cor}

Indeed, according to~Lemma~\ref{lem:out-degree} (see Section~\ref{sec:c_new}) a.a.s.~the out-degrees of all vertices do not exceed $\omega \log n$. Therefore, even if out-neighbours of a vertex form a complete graph, the contribution from them is at most $\binom{\omega \log n}{2}$, which is much smaller than the required lower bound for $k$.

Finally, let us discuss the connection between the obtained results and the empirical observations discussed in Section~\ref{sec:clustering}. Recall that it was observed that in real-world networks $C(d) \propto d^{-\varphi}$ for some $\varphi > 0$ and it is often the case that $\varphi = 1$. Basically, the results discussed in this section mean that in SPA model we have $\varphi = 1$ and other constants cannot be modeled. Indeed, Theorem~\ref{thm:average} (ii)' (and the corresponding Corollary~\ref{cor:undir}) state slightly weaker results than the asymptotics for $C(d)$, since instead of averaging over the vertices of degree $d$ we average over the set $X_d$ of vertices with degree close to $d$. On the other hand, Theorem~\ref{thm:average} (i)' states a stronger result for individual vertices.


\section{Simulations}\label{sec:simulations}

In this section, we illustrate the theoretical, asymptotic, results presented in the previous section by analyzing the local clustering coefficient for graphs of various orders generated according to the SPA model.
An efficient algorithm used to generate the SPA graphs is described in the proceeding version of this paper~\cite{iskhakov2018clustering}. 

It is proven in Theorem~\ref{thm:average} that $\frac{1}{|X_d|} \sum_{v \in X_d} c^-(v,n) = \Theta (1/d)$ for $d \gg \log^C n$, where $C = 4 + (4pA_1 + 2)/(pA_1(1-pA_1))$.
In order to illustrate this result, for each $p \in \{0.1,0.2,\ldots,0.9\}$, we generated 10 graphs with parameters $A_1 = 1$ and $A_2 = 10(1-p)/p$ ($A_2$ is chosen to fix the expected asymptotic degree equal 10) and computed the average value of $C^-(d,n)$ for $n=10^6$, see Figure~\ref{fig:average_clustering} (left). Similarly, Figure~\ref{fig:average_clustering} (right) presents the same measurements for the undirected average local clustering $C(d,n)$.
Note that in both cases figures agree with our theoretical results: both $C^-(d,n)$ and $C(d,n)$ decrease as $c/d$ with some $c$ for large enough $d$ (we added a function $10/d$ for comparison).
Note that for small $p$ the maximum degree is small, therefore the sizes of the generated graphs are not large enough to observe a straight line in log-log scale. 

\begin{figure}
\begin{center}
\includegraphics[width=0.49\textwidth]{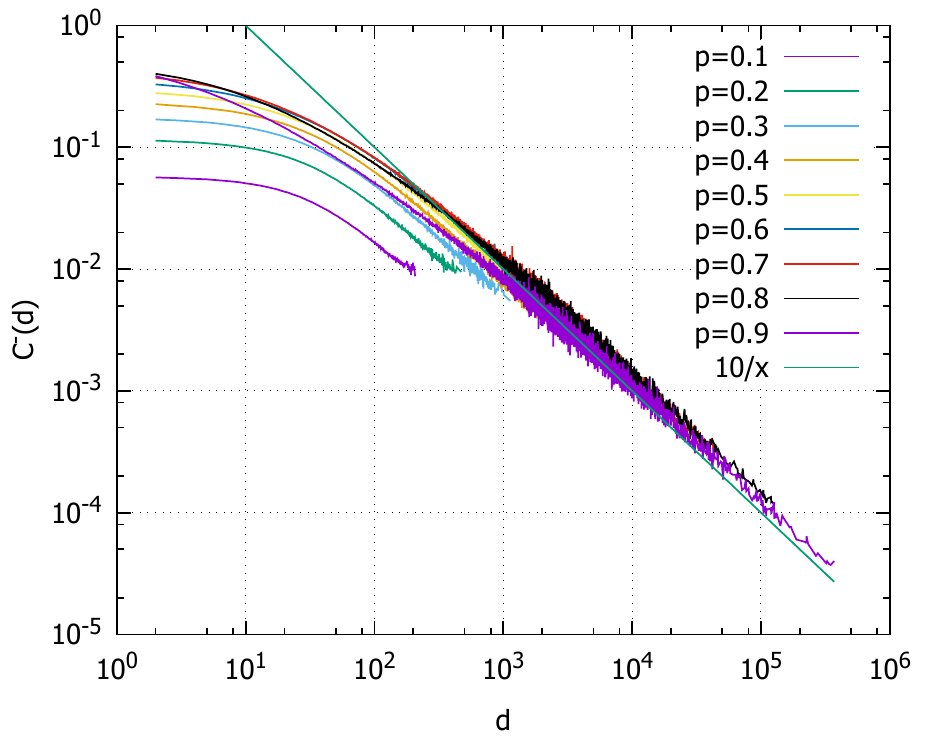}
\includegraphics[width=0.49\textwidth]{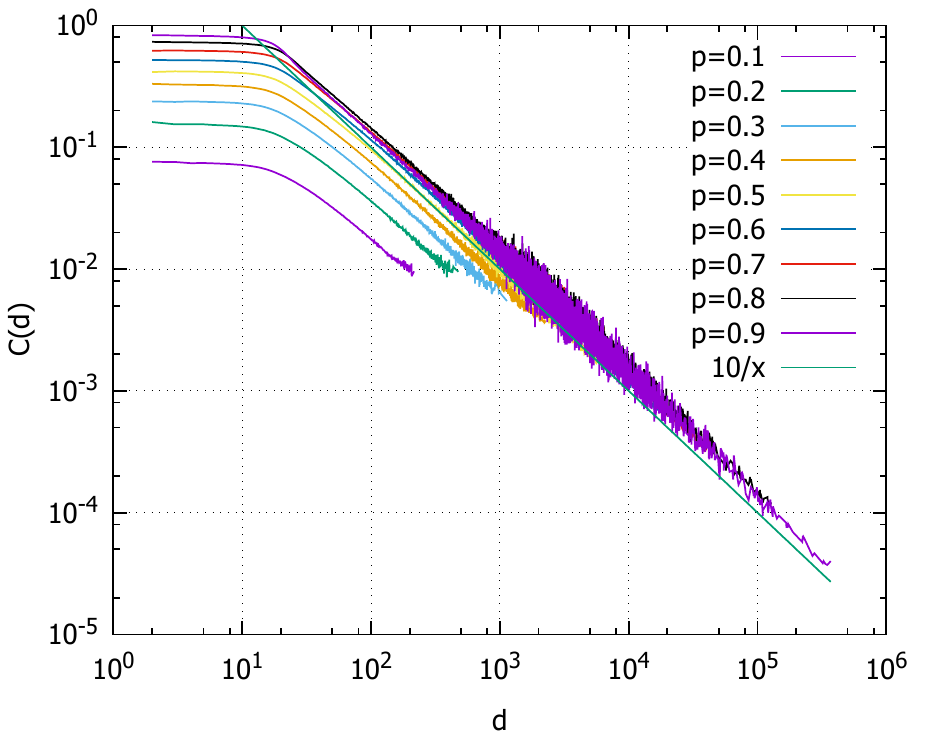}
\vspace{-10pt}
\caption{Average local clustering coefficient for directed (left) and undirected (right) graphs.}
\label{fig:average_clustering}
\end{center}
\vspace{-20pt}
\end{figure}

Note that for all $p\in (0,1)$ we have $C = 4 + \frac{4p+2}{p(1-p)} > 18$, so, our theoretical results are expected to hold for $d \gg \log^C n > 10^{20}$ which is irrelevant as the order of the graph is only $10^{6}$. However, we observe the desired behaviour for much smaller values of $d$; that is, in some sense, our bound is too pessimistic.
 

Also, note that the statement $C^-(d,n) = \Theta(1/d)$ is stronger than the statement $C^-(X_d) = \Theta(1/d)$ of Theorem~\ref{thm:average}, since in the theorem we averaged $c^{-}(v,n)$ over the set $X_d$ of vertices of in-degree between $(1-\delta)d$ and $(1+\delta)d$. In order to illustrate the difference, on Figure~\ref{fig:smooth_clustering} we present the smoothed curves for the directed (left) and undirected (right) local clustering coefficients averaged over $X_d$ for $\delta = 0.1$. Note that this smoothing substantially reduce the noise observed on Figure~\ref{fig:average_clustering}.

\begin{figure}
\begin{center}
\includegraphics[width=0.49\textwidth]{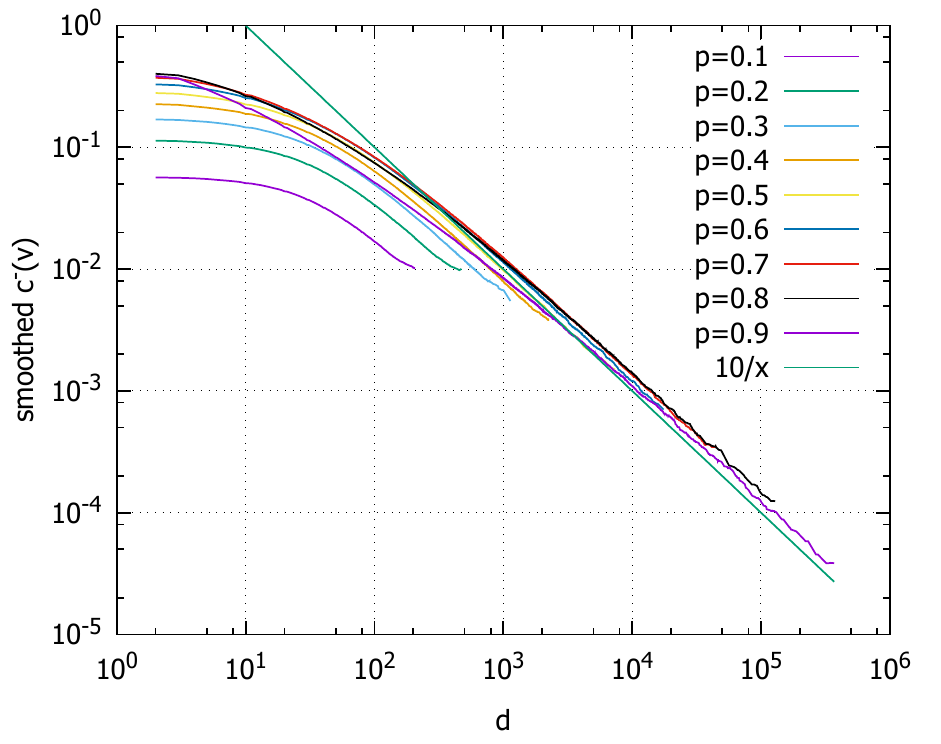}
\includegraphics[width=0.49\textwidth]{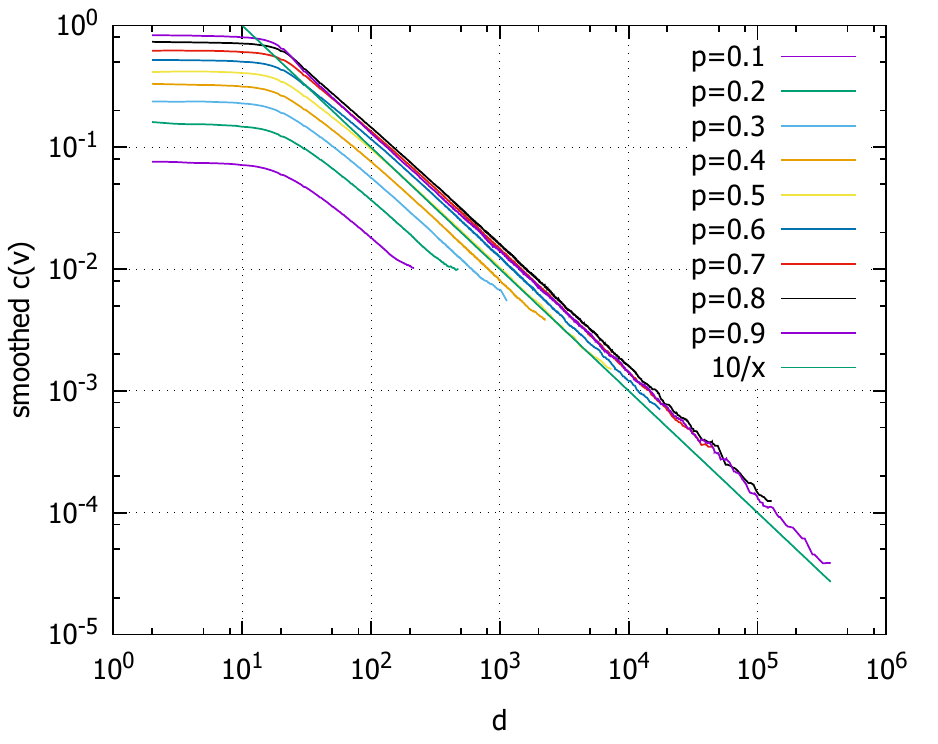}
\vspace{-10pt}
\caption{Local clustering coefficient for directed (left) and undirected (right) graphs averaged over $X_d$.}
\label{fig:smooth_clustering}
\end{center}
\vspace{-20pt}
\end{figure}

Next, let us illustrate the fact that the number of edges between ``new'' neighbours of a vertex is more predictable than the number of edges going from some neighbours to ``old'' ones. We extensively used this difference in Section~\ref{sec:results}, where we analyzed new and old edges separately. 
In our experiments, we split $c^-(v,n)$ into ``old'' and ``new'' parts as in Equation~\eqref{eq:old_new}, but now we take $\hat{T}_v$ be the smallest integer $t$ such that $\deg^-(v,t)$ exceeds $\deg^-(v,n)/2$. As a result, we compute the average local clustering coefficients $C_{old}^-(d)$ and $C_{new}^-(d)$. Figure~\ref{fig:old_new} shows that $C_{new}^-(d)$ can almost perfectly be fitted by $c/d$ with some $c$, while most of the noise comes from  $C_{old}^-(d)$. 

\begin{figure}
\begin{center}
\includegraphics[width=0.49\textwidth]{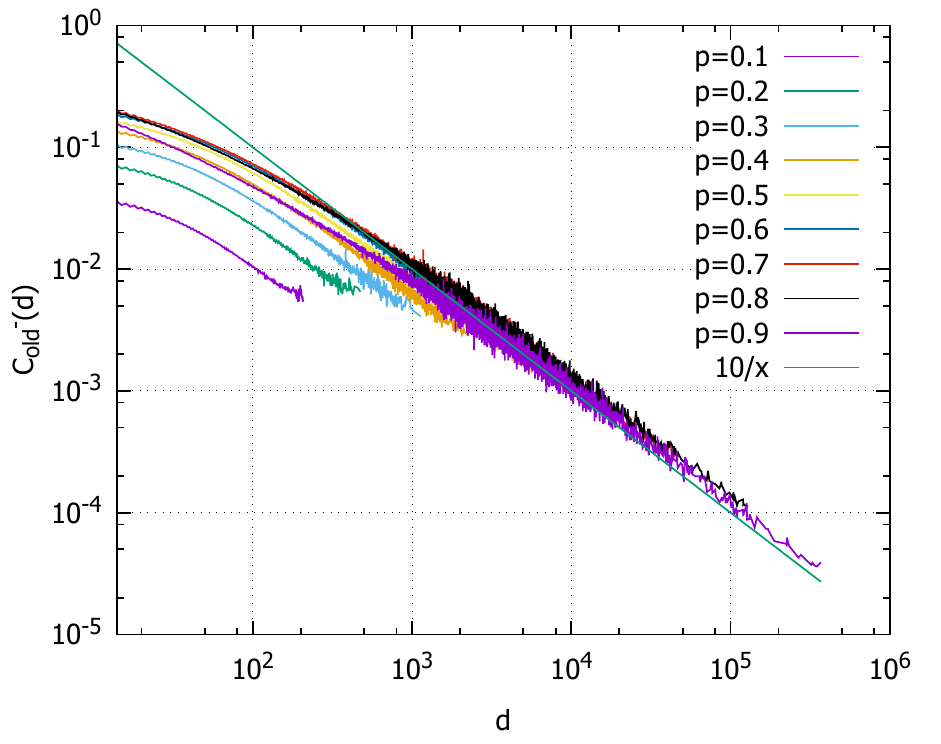}
\includegraphics[width=0.49\textwidth]{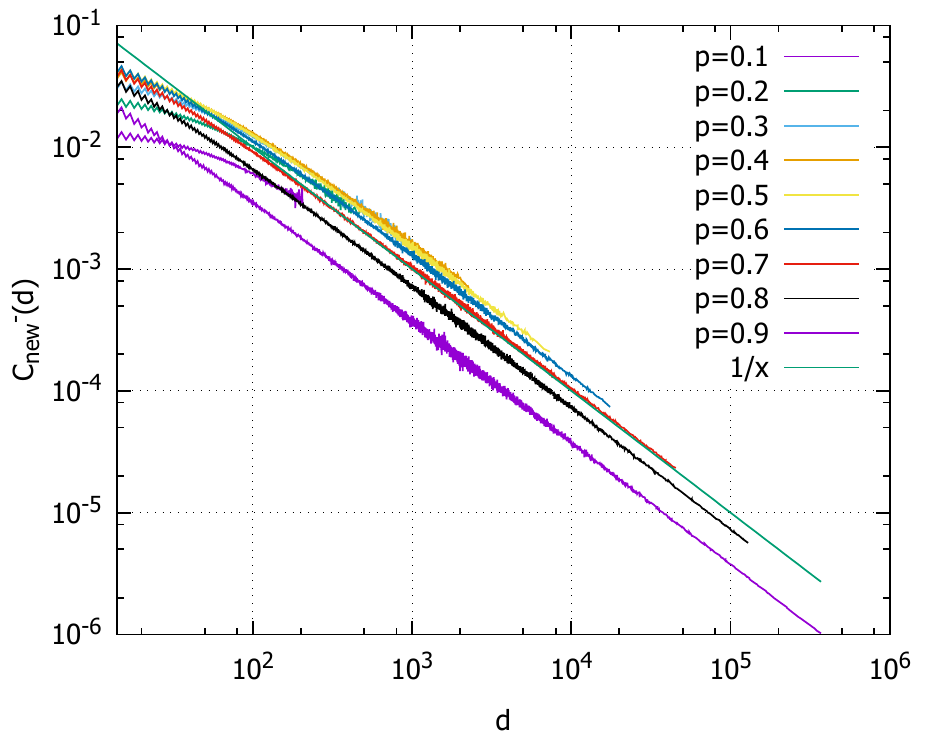}
\caption{Comparison of ``new'' and ``old'' parts of the average local clustering coefficient.}
\label{fig:old_new}
\end{center}
\end{figure}

Finally, Figure~\ref{fig:individual} shows the distribution of the individual local clustering coefficients for one graph generated with $p=0.7$. Theorem~\ref{thm:negative} states that a.a.s. there exist a vertex $v$ of degree $d$ with $c^-(v,n) \gg 1/d$. Also, according to this theorem, the situation is much worse for smaller values of $d$. 
Indeed, one can see on Figure~\ref{fig:individual} that for small $d$ the scatter of points is much larger. On the other hand, in Theorem~\ref{thm:average} we present bounds for $c^-(v,n)$ for almost all vertices, provided that $d$ is large enough. One can see it on the figure too and, similarly to the previously discussed figures, we observe the expected behaviour even for relatively small $n$ despite the bound $\log^C n$ that is bigger than $n$ in our case.

\begin{figure}
\begin{center}
\includegraphics[width=0.6\textwidth]{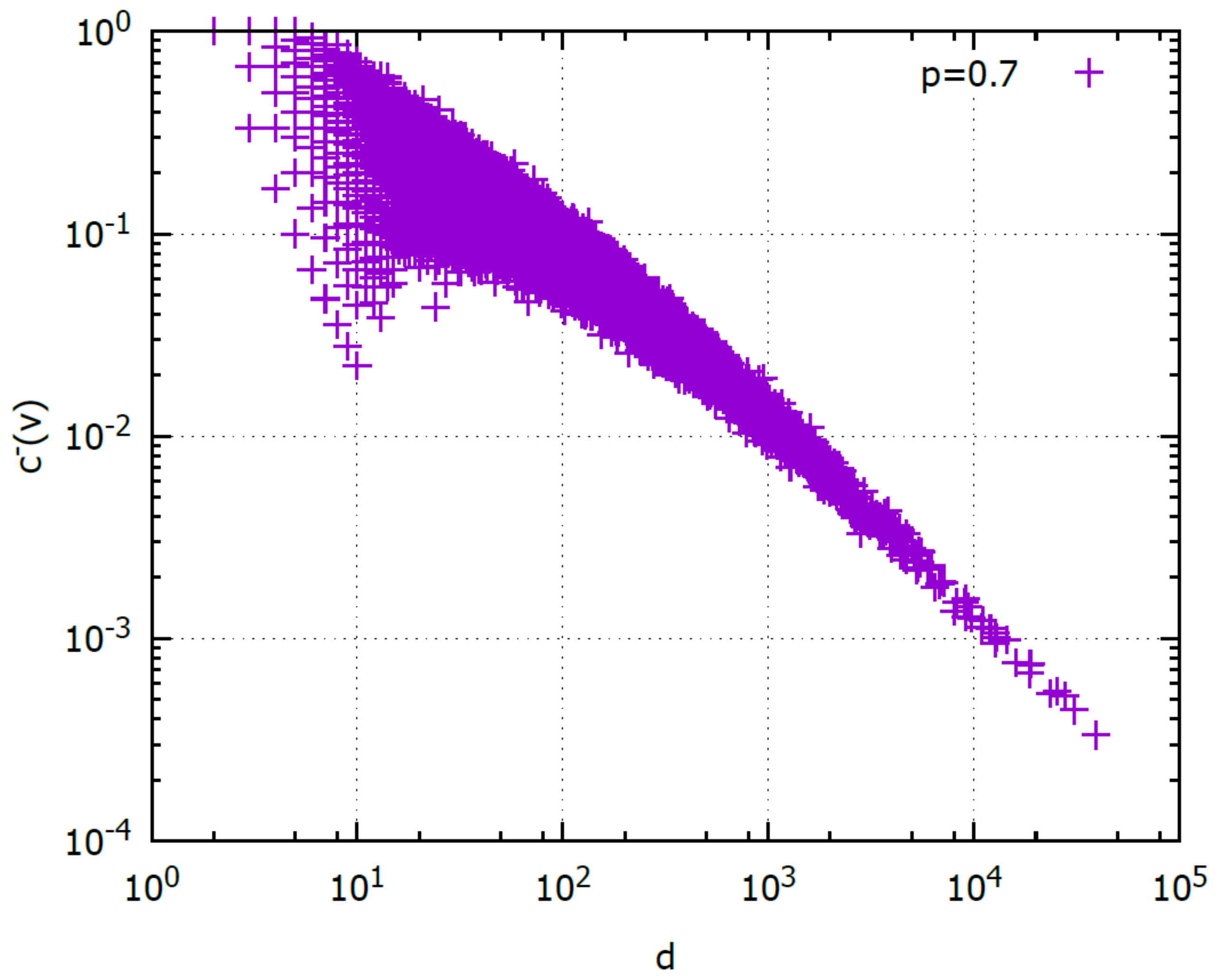}
\caption{The distribution of individual local clustering coefficients.}
\label{fig:individual}
\end{center}
\end{figure}

\section{Proofs}\label{sec:proofs}

\subsection{Proof of Theorem~\ref{degconc}} \label{sec:proof-degconc}

We will use the following version of the Chernoff bound that can be found, for example, in~\cite[p.\ 27, Corollary~2.3]{JLR}. 

\begin{lemma}\label{lem:Chernoff} 
Let $X$ be a random variable that can be expressed as a sum of independent random indicator variables, $X=\sum_{i=1}^{n} X_{i}$, where $X_{i}\in\mathrm{Ber}(p_{i})$ with (possibly) different $p_{i}=\mathbb{P} (X_{i} = 1)= \mathbb{E}X_{i}$. If $\varepsilon\le3/2$, then
\begin{align}\label{eq:Ch}
\mathbb{P} (|X - \mathbb{E }X| \ge\varepsilon\mathbb{E }X)  &  \le2 \exp\left(  - \frac{\varepsilon^{2} \mathbb{E }X}{3} \right)  .
\end{align}
\end{lemma}

\bigskip

Let us start with the following key lemma. 

\begin{lemma}\label{thm:conc}
Let $\omega=\omega(n)$ be any function tending to infinity together with $n$.  For a given vertex $v$, suppose that $\deg^-(v,T)=d \ge \omega \log n$. Then, with probability $1-o(n^{-6})$, for every value of $t$, $T \le t  \le 2T$, 
$$
\left| \deg^-(v,t) - d \cdot \left( \frac tT \right)^{p A_1} \right| \le \frac {5}{p A_1} \cdot \frac {t}{T} \sqrt{d \log n}.
$$
\end{lemma}

Of course, we will be applying the lemma to the graph generated by the SPA model. However, the statement is much more general; that is, one can apply it to any (deterministic) graph on $T$ vertices as long as the in-degree of $v$ in this graph is at least $\omega \log n$ and the next $T$ vertices are added according to the rules of the SPA model. After that, the desired bounds for the in-degree of $v$ hold with probability $1-o(n^{-6})$.

Before we move to the proof of the lemma, let us recall a standard, but very useful proof technique in probability theory that allows one to compare two random variables. Consider two biased coins, the first with probability $p$ of turning up heads and the second with probability $q > p$ of turning up heads. For any fixed $k$, the probability that the first coin produces at least $k$ heads should be less than the probability that the second coin produces at least $k$ heads. However, proving it is rather difficult with a standard counting argument. Coupling easily circumvents this problem. Let $X_1, X_2, \ldots, X_n$ be indicator random variables for heads in a sequence of $n$ flips of the first coin. For the second coin, define a new sequence $Y_1, Y_2, \ldots, Y_n$ such that if $X_i = 1$, then $Y_i = 1$; if $X_i = 0$, then $Y_i = 1$ with probability $(q-p)/(1-p)$. Clearly, the sequence of $Y_i$ has exactly the probability distribution of tosses made with the second coin. However, because of the coupling we trivially get that $X := \sum X_i \le Y:= \sum Y_i$ and so $\Prob(X \ge k) \le \Prob(Y \ge k)$, as expected. We will say that $X$ is (stochastically) bounded from above by $Y$.

\begin{proof}
Our goal is to estimate $\deg^-(v,t) - d \cdot \left( t/T \right)^{p A_1}$. We will show that the upper bound holds; the lower bound can be obtained by using an analogous, symmetric, argument.


Let $X_{T}, X_{T+1}, \ldots, X_{2T-1}$ be a sequence of independent Bernoulli random variables (with various parameters):
$$
\Prob(X_t = 1) = p \ \frac {A_1 \left( d \left( \frac tT \right)^{p A_1} + \frac {5}{p A_1} \cdot \frac {t}{T} \sqrt{d \log n} \right) + A_2}{t}.
$$
We will use the following standard coupling. Let $Z_T = \deg^-(v,T) = d$. For each $T < t \le 2T$ we define $Z_t = \deg^-(v,t)$ if
$$
\deg^-(v,t') \le d \cdot \left( \frac {t'}{T} \right)^{p A_1} + \frac {5}{p A_1} \cdot \frac {t'}{T} \sqrt{d \log n}
$$
for every $T \le t' < t$; otherwise, $Z_t = Z_{t-1} + X_{t-1}$. It follows that $Z_t-d$ can be (stochastically) bounded from above by the sum $X'_t = \sum_{i=T}^{t-1} X_i$ of independent indicator random variables. Note that if $T_0$ is the smallest $t$ for which $\deg^-(v,t) > d \cdot \left( \frac {t}{T} \right)^{p A_1} + \frac {5}{p A_1} \cdot \frac {t}{T} \sqrt{d \log n}$, then the following stochastic bound holds: $\deg^-(v,t)=Z_t \le d + X'_t$ for all $t \le T_0$. Hence, the lemma will follow once we show that with the desired probability 
\begin{equation}\label{eq:desired_X'_t}
d + X'_t \le d \cdot \left( \frac {t}{T} \right)^{p A_1} + \frac {5}{p A_1} \cdot \frac {t}{T} \sqrt{d \log n}
\end{equation}
for all $T \le t \le 2T$.

Clearly, since  $p A_1 < 1$,
\begin{eqnarray*}
\E X'_t &=& \sum_{i=T}^{t-1} \E X_i \\
&=& p A_1 d T^{-p A_1} \left( \sum_{i=T}^{t-1} i^{p A_1-1} \right) + \frac {t-T}{T} 5 \sqrt{d \log n} + O(1) \\
&=& d \left( \frac {t}{T} \right)^{p A_1} - d\left( \frac {T}{T} \right)^{p A_1} + \frac {t-T}{T}  5 \sqrt{d \log n} + O(1)\\
&=& d \left( \frac {t}{T} \right)^{p A_1} - d + \frac {t-T}{T}  5 \sqrt{d \log n} + O(1).
\end{eqnarray*}
If~(\ref{eq:desired_X'_t}) fails, then
\begin{eqnarray*}
X'_t &\ge& \left( d \cdot \left( \frac{t}{T} \right)^{p A_1} + \frac {5}{p A_1} \cdot \frac {t}{T} \sqrt{d \log n} \right) - d \\
&=& \E X'_t + \frac {5}{p A_1} \cdot \frac {t}{T} \sqrt{d \log n} - \frac {t-T}{T}  5 \sqrt{d \log n} + O(1) \\
&=&  \E X'_t + 5 \sqrt{d \log n} + 5 \left( \frac {1}{p A_1} - 1 \right) \frac {t}{T} \sqrt{d \log n} + O(1)  \\
&\ge& \E X'_t + 5 \sqrt{d \log n}, 
\end{eqnarray*}
using again that it is assumed that $p A_1 < 1$ and so $1/(pA_1)-1 > 0$. It follows from the Chernoff bound~(\ref{eq:Ch}) that
$$
\Prob (|X'_t - \E X'_t| \ge 5 \sqrt{d \log n}) \le 2 \exp \left(  - \ (5\eps/3) \sqrt{d \log n} \right),
$$
where $\eps = 5 \sqrt{d \log n} / \E X'_t$. The maximum value of $\E X'_t$ corresponds to $t=2T$ and so 
\begin{eqnarray*} 
\E X &\leq& d \left( \frac {2T}{T} \right)^{p A_1} - d + \frac {2T-T}{T}  5 \sqrt{d \log n} + O(1) \\
&\sim& d (2^{p A_1} - 1) ~~\le~~ d. 
\end{eqnarray*}
So $\eps \ge 5 \sqrt{d^{-1} \log n}$. Therefore, the probability that~(\ref{eq:desired_X'_t}) fails for a given $T \le t \le 2T$ is at most $2 \exp (- (25/3) \log n) = o(n^{-7})$. Hence,~(\ref{eq:desired_X'_t}) fails for some $T \le t \le 2T$ with probability $o(n^{-6})$ and the proof is finished.
\end{proof}

\bigskip

Now, with Lemma~\ref{thm:conc} in hand we can get Theorem~\ref{degconc}. 

\begin{proof}[Proof of Theorem~\ref{degconc}]
Let $\omega = \omega(n)$ be a function going to infinity with $n$. Let $v$ be a vertex with final degree $k\geq \omega\log n$. Let $T$ be the first time  that the in-degree of $v$ exceeds $(\omega/2) \log n$. Finally, let $d = \deg^-(v,T)$. We obtain from Lemma~\ref{thm:conc} that, with probability $1-o(n^{-6})$, 
\begin{multline*}
d \left( \frac tT \right)^{p A_1} \left(1 - \frac {10}{p A_1} \sqrt{d^{-1} \log n} \right) \le \deg^-(v,t) \\ \le d \left( \frac tT \right)^{p A_1} \left(1 + \frac {10}{p A_1} \sqrt{d^{-1} \log n} \right)
\end{multline*}
for $T \le t \le 2T$. It follows that the degree tends to grow but the sphere of influence tends to shrink between $T$ and $2T$, and thus that the conditions of Lemma \ref{thm:conc} again hold at time $2T$. We can now keep applying the same lemma for times $2T$, $4T$, $8T$, $16T, \dots$, using the final value as the initial one for the next period, to get the statement for all values of $t$ from $T$ up to and including time $n$. Precisely, for $1\leq i<i_{\max}=\lfloor \log_2 n\rfloor+1$, let $d_i=\deg^-(v,2^i T)$. Then by Lemma~\ref{thm:conc}, we have for $i>1$ that $d_i\leq d_{i-1}2^{p A_1} (1+\eps_i)$, where $\eps_i=\frac {10}{p A_1} \sqrt{d_{i-1}^{-1} \log n} $. Since we apply the lemma $O(\log n)$ times (for a given vertex $v$), the following statement holds with probability $1-o(n^{-5})$ from time $T$ on: for any $2^{i-1}T\leq t<2^{i} T$, we have that 
\[
\deg^-(v,t)\leq d \left( \frac tT \right)^{p A_1}\prod_{j=0}^i (1+\eps_i).
\]
It remains to make sure that the accumulated multiplicative error term is still only $(1+o(1))$.
For that, let us note that 
\begin{eqnarray*}
\prod_{j=0}^i (1+\eps_i) &=& \prod_{j=1}^{i} \left( 1 + \frac{10}{p A_1} \sqrt{d^{-1} 2^{-p A_1j} \log n} \right) \\
&\sim& \exp \left( \frac{10}{p A_1} \sqrt{d^{-1} \log n} \sum_{j=1}^{i} 2^{-p A_1j/2} \right) \\
&=& \exp \left( O(\sqrt{d^{-1} \log n}) \right) ~~\sim~~ 1,
\end{eqnarray*}
since $d$ grows faster than $\log n$. A symmetric argument can be used to show a lower bound for the error term and so the result holds. 

It follows that we have the desired behaviour from time $T$. Precisely, for times $T\leq t\leq n$, we have that 
\begin{equation*}
\deg^-(v,t) \sim d \left( \frac tT \right)^{p A_1},
\end{equation*}
where $d=\deg^-(v,T) \sim (\omega/2) \log n$.  Setting $t=n$ and $\deg^-(v,n)=k$, we obtain that 
\[
T \sim \left(\frac{d}{k}\right)^{1/p A_1}n \sim \left(\frac{\omega\log n}{2k}\right)^{1/p A_1}n \sim \left(\frac12\right)^{1/p A_1} T_v. 
\]
Therefore, for large enough $n$, we have that $T<T_v$. As a result, we obtain that, for $T_v \leq t\leq n$,
\begin{equation*}
\deg^-(v,t) \sim k \left( \frac{t}{n} \right)^{p A_1}.
\end{equation*}
Finally, since the statement holds for any vertex $v$ with probability $1-o(n^{-5})$, with probability $1-o(n^{-4})$  the statement holds for all vertices. The proof of the theorem is finished.
\end{proof}

Let us note that Theorem~\ref{degconc} immediately implies the following two corollaries. 
\begin{cor}\label{cor:tT}
Let $\omega =\omega (n)$ be any function tending to infinity together with $n$. The following holds with probability $1-o(n^{-4})$. For every vertex $v$, and for every time $T$ so that $\deg^-(v,T)\geq \omega\log n$, for all times $t$, $T\leq t\leq n$,
\[
\deg^-(v,t) \sim \deg^-(v,T) \left( \frac tT \right)^{p A_1}.
\]
\end{cor}

\begin{cor}\label{cor:upper}
Let $\omega =\omega (n)$ be any function tending to infinity together with $n$. The following holds with probability $1-o(n^{-4})$. For any vertex $v_i$ born at time $i \ge 1$, and $i \le t \le n$ we have that
\begin{equation}\label{upperbounddegree}
\deg^-(v_i,t)\leq \omega \log n \left(\frac{t}{i}\right)^{p A_1}.
\end{equation}
\end{cor}

\subsection{Proof of Theorem~\ref{thm:c_new}} \label{sec:c_new}

Let $B$ be a ball of volume $b = b(n)$ and $t = t(n) \in \nat$ be any function of $n$ such that $bt \to \infty$ as $n \to \infty$. It will be crucial for the argument to understand the behaviour of the random variables $N_{i,t} = N_{i,t}(b)$ counting the number of vertices in $B$ that are of in-degree $i$ at time $t$; that is,
$$
N_{i,t} = | \{ w \in B : \deg^-(w,t) = i \} |.
$$
The arguments presented below are similar to the ones in~\cite{spa1} showing that the degree distribution of $G_n$ follows a power-law. 

The equations relating the random variables $N_{i,t}$ are described as follows. As $G_0$ is the null graph, $N_{i,0}=0$ for $i \ge 0.$ For all $t \in \nat \cup \{0\}$, we derive that
\begin{eqnarray}
{\mathbb{E}}(N_{0,t+1}-N_{0,t}~|~G_{t}) &=& b-N_{0,t}p\frac{A_2} {t}, \label{unscaled} \\
{\mathbb{E}}(N_{i,t+1}-N_{i,t}~|~G_{t}) &=&N_{i-1,t}p\frac{A_1(i-1)+A_2}{t}-N_{i,t}p\frac{A_1i+A_2}{t}. \label{unscaled1}
\end{eqnarray}

Recurrence relations for the expected values of $N_{i,t}$ can be derived by taking the expectation of the above equations. To solve these relations, we use the following lemma on real sequences, which is Lemma~3.1 from~\cite{FCLL}.

\begin{lemma}\label{lem}
If $(\alpha_{t}),$ $(\beta_{t})$ and $(\gamma_{t})$ are real sequences satisfying the relation
\[
\alpha_{t+1}=\left( 1-\frac{\beta_{t}}{t}\right) \alpha_{t}+\gamma_{t},
\]
and $\lim_{t\rightarrow \infty }\beta_{t}=\beta>0$ and $\lim_{t\rightarrow \infty }\gamma_{t}=\gamma,$ then $\lim_{t\rightarrow \infty}\frac{\alpha_{t}}{t}$ exists and equals $\frac{\gamma}{1+\beta}$.
\end{lemma}

Applying this lemma with $\alpha_t=\E(N_{0,t}) / b$, $\beta_t=pA_2,$ and $\gamma_t=1$ gives that $\E(N_{0,t}) \sim c_0 b t$ with 
$$
c_0 = \frac {1}{1+pA_2}. 
$$ 
For $i \ge 1$, the lemma can be inductively applied with $\alpha_t=\E(N_{i,t}) / b$, $\beta_t=p(A_1i+A_2)$, and $\gamma_t=\E(N_{i-1,t}) p (A_1(i-1)+A_2) / (bt)$ to show that $\E(N_{i,t}) \sim c_ibt$, where
$$
c_{i}=c_{i-1} p \frac{A_1(i-1)+A_2}{1+p(A_1i+A_2)}.
$$
It is straightforward to verify that 
$$
c_i={\frac{p^i}{1+pA_2+ipA_1} \prod_{j=0}^{i-1}\frac{jA_1+A_2}{1+pA_2+jpA_1}}. 
$$
The above formula implies that $c_i=(1+o(1))ci^{-(1+1/(pA_1))}$ (as $i \to \infty$) for some constant $c$, so the expected proportion $N_{i,t}/(bt)$ asymptotically follows a power-law with exponent $1+1/(pA_1)$. 

\bigskip

We prove concentration for $N_{i,t}$ when $i\leq i_f$ (for some function $i_f = i_f(n)$) by using a relaxation of Azuma-Hoeffding martingale techniques. The random variables $N_{i,t}$ do not a priori satisfy the $c$-Lipschitz condition: indeed, a new vertex may fall into many overlapping regions of influence and so it can potentially change degrees of many vertices. Nevertheless, we will prove that deviations from the $c$-Lipschitz condition occur with very small probability. The following lemma gives a deterministic bound for $|N_{i,t+1}-N_{i,t}|$ which holds with high probability. Indeed, it is obvious that $|N_{i,t+1}-N_{i,t}| \le \max \{ \deg^+(v_{t+1}, t+1), 1\}$. Note that a weaker bound of $\log^2 n$ was proved in~\cite{spa1}; with Corollary~\ref{cor:upper} in hand, we can get slightly better bound but the argument remains the same. We present the proof for completeness.

\begin{lemma}\label{lem:out-degree}
Let $\omega =\omega (n)$ be any function tending to infinity together with $n$. The following holds with probability $1-o(n^{-3})$. For every vertex $v_i$, 
\[
\deg^+(v_i,i) = \deg^+(v_i,n) \le \omega \log n.
\]
\end{lemma}
\begin{proof}
Let us focus on any $1 \le i \le n$. Since $v_i$ is chosen uniformly at random from the unit hypercube (note that the history of the process does not affect this distribution) with the torus metric, without loss of generality, we may assume that $v_i$ lies in the centre of the hypercube. For $1 \le j<i$, let $X_j$ denote the indicator random variable of the event that $v_j$ lies in the ball around $v_i$ (or vice versa) with volume
$$
\alpha_j = j^{-pA_1} i^{pA_1-1} \omega^{2/3} \log n.
$$

By Corollary~\ref{cor:upper} (applied with $\omega^{1/3}$ instead of $\omega$), we may assume that 
$$
\deg^-(v_j,i) \le (i/j)^{pA_1} \omega^{1/3} \log n,
$$ 
for all $j \in [i-1]$. Note that $(A_1 \deg^-(v_j,i) + A_2)/i = o(\alpha_j)$. Hence, for all $j \in [i-1]$, $X_j = 0$ implies that $v_i$ is not in the influence region of $v_j$ and so there is no directed edge from $v_i$ to $v_j$. Therefore, we have that 
$$
\deg^+(v_i,i) \le \sum_{j=1}^{i-1} X_j.
$$ 
Since
\begin{multline*}
\E \left( \sum_{j=1}^{i-1} X_j \right) = \sum_{j=1}^{i-1} \alpha_j = i^{pA_1-1} \omega^{2/3} \log n \sum_{j=1}^{i-1} j^{-pA_1} \\ = O( \omega^{2/3} \log n) = o (\omega \log n),
\end{multline*}
the assertion follows easily from the Chernoff bound. 
\end{proof}

Now, we are ready to prove concentration for the random variables $N_{i,t}$. In order to explain the technique, we investigate $N_{0,t}$, the number of vertices of in-degree zero. The argument easily generalizes to other values of $i$ and we explain it afterwards. We will use the supermartingale method of Pittel et al.~\cite{PSW}, as described in~\cite{NW-DE}.

\begin{lemma}\label{l:corollary}
Let $G_{0},G_{1},\dots, G_{n}$ be a random graph process and $X_{t}$ a random variable determined by $G_{0},G_{1},\dots, G_{t}$, $0\leq t\leq n$. Suppose that for some real constants $\beta_t$ and constants $\gamma_{t}$,
$$
\E(X_{t}-X_{t-1} ~~|~~ G_{0},G_{1},\dots, G_{t-1})<\beta_t $$
and
$$
|X_{t}-X_{t-1}-\beta_t|\leq \gamma_{t}
$$
for $1\leq t\leq n$. Then for all $\alpha >0$,
$$
\mathbb{P}\left(\mbox{For some }s\ \mbox{with }0\leq s\leq n:X_{s}-X_{0}\geq \sum_{t=1}^s \beta_t +\alpha \right) \leq \exp \Big(-\frac{\alpha^{2}}{2\sum \gamma_{t}^{2}}\Big)\;.
$$
\end{lemma}

Now, we are ready to prove the concentration for $N_{0,t}$.

\begin{theorem}\label{thm:X_1}
Let $B$ be a ball of volume $b = b(n)$ and $t = t(n) \in \nat$ be any function of $n$ such that $bt \to \infty$ as $n \to \infty$. Let $\omega =\omega (n)$ be any function tending to infinity together with $n$. The following holds with probability $1-o(n^{-3})$.
\begin{multline*}
N_{0,t} = N_{0,t}(B) = \frac {bt}{1+A_2p} + O((bt)^{1/2} (\omega \log n)^{3/2}) \\ = c_0 bt + O((bt)^{1/2} (\omega \log n)^{3/2}).
\end{multline*}
In particular, if $bt \gg \log^3 n$, then $N_{0,t} \sim c_0 bt$.
\end{theorem}

\begin{proof}
We first need to transform $N_{0,s}$ ($1 \le s \le t$) into something close to a martingale. It provides some insight if we define real function $f(x)$ to model the behaviour of the scaled random variable $N_{0,xt}/t$. If we presume that the changes in the function correspond to the expected changes of the random
variable (see~(\ref{unscaled})), we obtain the following differential equation
$$
f'(x) = b - f(x) \frac {pA_2}{x}
$$
with the initial condition $f(0)=0$. The general solution of this equation can be put in the form
$$
f(x) x^{p A_2} - \frac {b x^{1+pA_2}}{1+pA_2} = C.
$$
Consider the following real-valued function
\begin{equation}
H(x,y)=x^{pA_2}y-\frac{b x^{1+pA_2}}{1+pA_2}  \label{e:H}
\end{equation}
(note that we expect $H(s,N_{0,s})$ to be close to zero). Let $\mathbf{w}_{s}=(s,N_{0,s})$, and consider the sequence of random variables $(H(\mathbf{w}_{s}) : 1\leq s\leq t).$ The second-order partial derivatives of $H$ evaluated at $\mathbf{w}_s$ are all $O(s^{pA_2-1})$. Moreover, it follows from Lemma~\ref{lem:out-degree} that we may assume that 
\begin{equation}
|N_{0,s+1} - N_{0,s}| \le \omega \log n, \label{eq:max_change}
\end{equation}
and clearly $|N_{0,s+1} - N_{0,s}| \le s$ as there are at most $s$ isolated vertices at time $s$.
Therefore, we have
\begin{eqnarray}
H(\mathbf{w}_{s+1})-H(\mathbf{w}_{s}) &=& (\mathbf{w}_{s+1}-\mathbf{w}_{s})\cdot \mathrm{grad}\mbox{ } H(\mathbf{w}_{s}) \nonumber \\
&& + O \Big( \min \{ s^{pA_2-1} \omega^2 \log^2 n, s^{pA_2+1} \} \Big), 
\label{e:h_diff}
\end{eqnarray}
where \textquotedblleft $\cdot $\textquotedblright\ denotes the inner product and $\mathrm{grad}\mbox{ } H(\mathbf{w}_{s})=(H_x(\mathbf{w}_s),H_y(\mathbf{w}_s))$.

Observe that from our choice of $H$, we have that
$$
\mathbb{E}(\mathbf{w}_{s+1}-\mathbf{w}_{s}~|~G_{s})\cdot \mbox{\ \textrm{grad} }H(\mathbf{w}_{s})=0.
$$
Hence, taking the expectation of~(\ref{e:h_diff}) conditional on $G_{s}$, we obtain that
$$
{\mathbb{E}}(H(\mathbf{w}_{s+1})-H(\mathbf{w}_{s})~|~G_{s})  = O(s^{pA_2-1} \omega^2 \log^2 n).
$$
From~(\ref{e:h_diff}) and~(\ref{eq:max_change}), noting that
$$
\mathrm{grad}\mbox{ }H(\mathbf{w}_{s})=\left(pA_2 s^{pA_2-1} N_{0,s}-bs^{pA_2}, s^{pA_2}\right),
$$
we have that
\begin{eqnarray*}
|H(\mathbf{w}_{s+1})-H(\mathbf{w}_{s}) -\beta_{s+1} | &=& O \Big(s^{pA_2} \omega \log n  + \min \{ s^{pA_2-1} \omega^2 \log^2 n, s^{pA_2+1} \}  \Big) \\
&=&  O (s^{pA_2} \omega \log n).
\end{eqnarray*}
(Indeed, if $s \le \omega \log n$, then $s^{pA_2+1} \le s^{pA_2} \omega \log n$; otherwise, $s^{pA_2-1} \omega^2 \log^2 n \le s^{pA_2} \log n$.)

Our goal is to apply Lemma~\ref{l:corollary} to the sequence $(H(\mathbf{w}_{s}):1\leq s\leq t)$ with 
\begin{eqnarray*}
\beta_{s} &=& C s^{pA_2-1} \omega^2 \log^2 n,\\
\gamma_{s} &=& C s^{pA_2} \omega \log n,
\end{eqnarray*}
where $C$ is a sufficiently large constant, to get an upper bound for $H(\mathbf{w}_{s})$. A symmetric argument applied to $(-H(\mathbf{w}_{s}):1\leq s\leq t)$ will give us the desired lower bound so let us concentrate on the upper bound. The bounds for $\beta_{s+1}$ and $\gamma_{s+1}$ we derived above are universal; however, typically vertex $v_s$  lies far away from the ball $B$ so that $N_{0,s}$ is not affected. This certainly happens if the distance from the ball $B$ to $v_s$ is more than the radius of the ball of volume $A_2/s$, and so this situation occurs with probability $1-O(b + s^{-1})$. Moreover, if this happens and $H(\mathbf{w}_s) \ge 0$, then $H(\mathbf{w}_s)$ decreases (it can be viewed as some kind of ``self-correcting'' behaviour); hence, since we aim for an upper bound, we may assume that $H(\mathbf{w}_s)$ does not change. It follows that 
\begin{eqnarray}
\sum_{s=1}^t \beta_s &=& O \left( \sum_{s=1}^t s^{pA_2-1} \omega^2 \log^2 n \cdot (b + s^{-1}) \right) \label{eq:beta} \\
&=& O \left( b\ \omega^2 \log^2 n \sum_{s=1}^t s^{pA_2-1}  \right) + O \left( \omega^2 \log^2 n \sum_{s=1}^t s^{pA_2-2}  \right) \nonumber \\
&=& O \left( b\ t^{pA_2} \omega^2 \log^2 n \right) + O \left( t^{pA_2-1} \omega^2 \log^2 n \right) = O \left( b\ t^{pA_2} \omega^2 \log^2 n \right), \nonumber
\end{eqnarray}
since it is assumed that $b t \to \infty$. Similarly, we get that
\begin{equation}
\sum_{s=1}^t \gamma_s^2 = O \left( \sum_{s=1}^t (s^{pA_2} \omega \log n)^2 \cdot (b + s^{-1}) \right) = O \left( b\ t^{1+2pA_2} \omega^2 \log^2 n \right). \label{eq:gamma}
\end{equation}
Finally, we are ready to apply Lemma~\ref{l:corollary} with $\alpha = b^{1/2} t^{1/2+pA_2} (\omega \log n)^{3/2}$ to obtain that with probability $1-o(n^{-3})$, 
$$
|H(\mathbf{w}_{t})-H(\mathbf{w}_{0})| = O(\alpha) = O(b^{1/2} t^{1/2+pA_2} (\omega \log n)^{3/2}).
$$
As $H(\mathbf{w}_{0})=0$, it follows from the definition (\ref{e:H}) of the function $H$, that with the desired probability
\begin{equation*}
N_{0,t}=\frac{bt}{1+pA_2}+O((bt)^{1/2} (\omega \log  n)^{3/2}),
\end{equation*}
which finishes the proof of the theorem.
\end{proof}

We may repeat (recursively) the argument as in the proof of Theorem~\ref{thm:X_1} for $N_{i,t}$ with $i\geq 1$. Since the expected change for $N_{i,t}$ is slightly different now (see~(\ref{unscaled1})), we obtain our result by considering the following function:
$$
H(x,y)=x^{p(A_1 i +A_2)}y-c_{i-1} \frac{p (A_1 (i-1)+A_2)}{1+p(A_1 i + A_2)} x^{1+p(A_1 i + A_2)}.
$$
Moreover, in bounding $\sum \beta_s$ and $\sum \gamma_s^2$ (see~(\ref{eq:beta}) and~(\ref{eq:gamma})) we need $b$ to be of order at least $(A_1 i + A_2)/t$; say, $bt \gg i$.  
Other than these minor adjustments, the argument is similar as in the case $i=0$, and we get the following result. Note that the conclusion (the last claim) follows as 
$$
c_i bt = \Theta( i^{-1-1/(pA_1)} bt) = \Theta( i (bt)^{1/2} i^{-2-1/(pA_1)} (bt)^{1/2}) \gg i (bt)^{1/2} (\log n)^{3/2},
$$
provided $bt \gg i^{4+2/(pA_1)} \log^3 n$.

\begin{theorem}\label{thm:X_t}
Let $B$ be a ball of volume $b = b(n)$, $t = t(n) \in \nat$, and $i_f = i_f(n) \in \nat$ be any functions of $n$ such that $bt \gg i_f$. Let $\omega =\omega (n)$ be any function tending to infinity together with $n$. The following holds with probability $1-o(n^{-2})$. For any $0 \le i \le i_f$,
$$
N_{i,t} = N_{i,t}(B) = c_i bt + O(i (bt)^{1/2} (\omega \log n)^{3/2}).
$$
In particular, if $bt \gg i^{4+2/(pA_1)} \log^3 n$, then $N_{i,t} \sim c_i bt$.
\end{theorem}

Finally, we can move to the proof of Theorem~\ref{thm:c_new}. 

\begin{proof}[Proof of Theorem~\ref{thm:c_new}]
Let us fix any vertex $v$ for which 
$$
\deg^-(v,n)=k=k(n) \geq (\omega \log n)^{4 + (4pA_1+2)/(pA_1(1-pA_1))}.
$$
Based on Theorem~\ref{degconc}, as we aim for the statement that holds with probability $1-o(n^{-1})$, we may assume that for all values of $t$ such that 
$$
n \left(\frac{\omega \log n}{k}\right)^{\frac{1}{p A_1}} =: T_v \le t \le n,
$$
we have
$$
\deg^-(v,t) \sim k \left(\frac{t}{n}\right)^{p A_1}.
$$
For any $\ell \in \nat \cup \{0\}$, let 
$$
t_\ell = 2^\ell T_v, \qquad b_\ell = A_1 k  t_\ell^{pA_1-1} n^{-pA_1},
$$
$B_\ell$ be the ball around $v$ of volume $b_\ell$, and $L$ be the smallest integer $\ell$ such that $t_\ell \ge n$. In fact, we will assume that $t_L = n$, as we may adjust $\omega$ (that is, multiply by a constant in $(1/2,1)$), if needed. Let $t_v := n (\omega \log n)^{-1/(pA_1)}$; since $k \ge (\omega \log n)^2$, we have $T_v \le t_v \le n$. Let $L'$ be the smallest integer $\ell$ such that $t_\ell \ge t_v$. 

Times $t_0 = T_v$, $t_{L'} = \Theta(t_v)$, and $t_L = n$ are important stages of the process; vertex $v$ has, respectively, degree $(1+o(1)) \omega \log n$, $\Theta(k / (\omega \log n))$, and $k$. Note that at time $t_\ell$ (for any $0 \le \ell \le L$) the sphere of influence of $v$ has volume $(1+o(1)) b_\ell$. Moreover, based on Corollary~\ref{cor:upper} (applied with, say, $\sqrt{\omega}$ instead of $\omega$) we may assume, since we aim for the statement that holds with probability $1-o(n^{-1})$, that any vertex $v_i$ born after time $T_v$ satisfies (for any $T_v \le t \le n$)
\begin{equation}
\deg(v_i,t) \le \sqrt{\omega} \log n \left( \frac ti \right)^{pA_1} = o\left( \omega \log n \left( \frac ti \right)^{pA_1} \right) = o(\deg(v,t)); \label{eq:bound_for_degree}
\end{equation}
as a result, the sphere of influence of $w$ has negligible volume comparing to the one of $v$. 

\bigskip

We will independently prove an upper bound and a lower bound of $c_{new}(v,n)$. In order to do it, we need to estimate $|E_{new}(N^-(v,n))|$, the number of directed edges from $u$ to $w$, where both $u$ and $w$ are neighbours of $v$ born after time $T_v$.

\medskip

\noindent \emph{Proof of $c_{new}(v,n) = O(1/k)$}: Suppose that a neighbour $w$ of $v$ lies in $B_{\ell-1} \setminus B_\ell$ for some $\ell$. An easy but an important observation is that at any time $t \ge t_{\ell+1}$, the sphere of influence of $v$ is completely disjoint from the one of $w$. Hence, the number of edges to $w$ that contribute to $c_{new}(v,n)$ can be upper bounded by $\deg^-(w, t_{\ell+1})$. It follows that
\begin{eqnarray*}
|E_{new}(N^-(v,n))| &\le& \sum_{\ell=1}^{L-2} \sum_{w \in B_{\ell-1} \setminus B_\ell} \deg^-(w, t_{\ell+1}) + \sum_{w \in B_{L-2}} \deg^-(w, t_{L}) \\
&\le& \sum_{\ell=1}^{L-1} \sum_{w \in B_{\ell-1}} \deg^-(w, t_{\ell+1}).
\end{eqnarray*}
Let $i_f = (\omega \log n)^{1/(1-pA_1)}$. We will independently deal with the largest balls, namely $B_{\ell}$ for $\ell < L'$; for the remaining ones, we will deal with vertices of degree more than $i_f$ before analyzing the contribution from low degree ones. In other words, we are going to show that each of the following three random variables is of order at most $k$:
\begin{eqnarray*}
\alpha &=& \sum_{\ell=1}^{L'-1} \sum_{w \in B_{\ell-1}} \deg^-(w, t_{\ell+1}), \\ 
\beta &=& \sum_{\ell=L'}^{L-1} \sum_{\substack{w \in B_{\ell-1}\\\deg^-(w, t_{\ell+1}) \le i_f}} \deg^-(w, t_{\ell+1}), \\ 
\gamma &=& \sum_{\ell=L'}^{L-1} \sum_{\substack{w \in B_{\ell-1}\\\deg^-(w, t_{\ell+1}) > i_f}} \deg^-(w, t_{\ell+1}).
\end{eqnarray*}
The conclusion will follow immediately as $|E_{new}(N^-(v,n))| \le \alpha + \beta + \gamma$.

In order to bound $\alpha$, we only need to use~(\ref{eq:bound_for_degree}). Let $E$ be the event that all the vertices satisfy~(\ref{eq:bound_for_degree}) which holds with probability $1-o(n^{-1})$. It follows that 
$$
\E ( \deg^-(v_i, t_{\ell+1}) ) \le \sqrt{\omega} \log n \left( \frac ti \right)^{pA_1} \cdot \Prob (E) + n \cdot \Prob(E^C) \sim  \sqrt{\omega} \log n \left( \frac ti \right)^{pA_1}
$$ 
and so
\begin{align*}
\E (\alpha) &= \sum_{\ell=1}^{L'-1} b_{\ell-1} \sum_{i = 1}^{t_{\ell+1}}  \E (\deg^-(v_i, t_{\ell+1}))  \\
&\le (1+o(1)) \sum_{\ell=1}^{L'-1} b_{\ell-1} \sum_{i = 1}^{t_{\ell+1}} \sqrt{\omega} \log n \left( \frac {t_{\ell+1}}{i} \right)^{pA_1}  \\ 
&= \sum_{\ell=1}^{L'-1} \sqrt{\omega} \log n \ b_{\ell-1} \ t_{\ell+1}^{pA_1} \sum_{i = 1}^{t_{\ell+1}} i^{-pA_1} \\
&= \sum_{\ell=1}^{L'-1} \Theta \left( \sqrt{\omega} \log n \ b_{\ell-1} \ t_{\ell+1} \right)  \\ 
&= \sum_{\ell=1}^{L'-1} \Theta \left( \sqrt{\omega} \log n \ k \ \left( \frac {t_{\ell+1}}{n} \right)^{pA_1} \right) \\
&= \Theta \left( \sqrt{\omega} \log n \ k \ \left( \frac {t_{L'}}{n} \right)^{pA_1} \right) \sum_{\ell=1}^{L'-1} 2^{-\ell pA_1} \\
&= \Theta \left( \sqrt{\omega} \log n \ k \ \left( \frac {t_{L'}}{n} \right)^{pA_1} \right) = o(k).
\end{align*}
The fact that, with the desired probability, $\alpha = O(k)$ follows from a standard martingale argument (for example, one could use Lemma~\ref{l:corollary}). 

Similarly, we can deal with $\gamma$. It follows from~(\ref{eq:bound_for_degree}) that no vertex born after time
\begin{multline*}
\left( \frac {\sqrt{\omega} \log n}{i_f} \right)^{1/(pA_1)} t_{\ell+1} \le (\sqrt{\omega} \log n)^{(1-1/(1-pA_1))/(pA_1)} \ t_{\ell+1} \\ = (\sqrt{\omega} \log n)^{-1/(1-pA_1)} \ t_{\ell+1}
\end{multline*}
can satisfy $\deg^-(w, t_{\ell+1}) > i_f$. Hence, 
\begin{multline*}
\E (\gamma) = \sum_{\ell=L'}^{L-1} b_{\ell-1} \sum_{i = 1}^{(\sqrt{\omega} \log n)^{\frac{-1}{1-pA_1}} t_{\ell+1}}  \E ( \deg^-(v_i, t_{\ell+1}) ) ~~\le~~ \sum_{\ell=L'}^{L-1} \Theta \left( b_{\ell-1} \ t_{\ell+1} \right)  \\ 
= \sum_{\ell=L'}^{L-1} \Theta \left( k \ \left( \frac {t_{\ell+1}}{n} \right)^{pA_1} \right) = \Theta \left( k \right) \sum_{\ell=0}^{L'-1} 2^{-\ell pA_1} = O(k).
\end{multline*}

Finally, we need to deal with $\beta$. This time, we need to use Theorem~\ref{thm:X_t} to count (independently) the number of vertices in $B_{\ell-1}$ of a certain degree. We may apply this theorem as for any $L' \le \ell \le L-1$, we have
\begin{eqnarray*}
b_{\ell-1} t_{\ell+1} &\ge& b_{L'-1} t_{L'+1} = \Theta \left( k \left( \frac {t_v}{n} \right)^{pA_1} \right) = \Theta \left( \frac {k}{\omega \log n} \right) \\
&=& \Omega \left( (\omega \log n)^{3 + (4pA_1+2)/(pA_1(1-pA_1))} \right) \\
&=& \Omega \left( (\omega \log n)^3 \ i_f^{(4pA_1+2)/(pA_1)} \right) \\
&\gg& i_f^{4+2/(pA_1)} \ \log^3 n,
\end{eqnarray*}
since $k \ge (\omega \log n)^{4 + (4pA_1+2)/(pA_1(1-pA_1))}$. 
(In fact, this is the main bottleneck that forces us to assume that $k$ is large enough.)
We get the following:
\begin{eqnarray*}
\beta &=& \sum_{\ell=L'}^{L-1} \sum_{i=1}^{i_f} i N_{i,t_{\ell+1}}(B_{\ell-1}) ~~=~~ (1+o(1)) \sum_{\ell=L'}^{L-1} \sum_{i=1}^{i_f} i c_i b_{\ell-1} t_{\ell+1} \\
&=& \Theta \left( \sum_{\ell=L'}^{L-1} b_{\ell-1} t_{\ell+1} \sum_{i=1}^{i_f} i^{-1/(pA_1)} \right) ~~=~~ \Theta \left( \sum_{\ell=L'}^{L-1} b_{\ell-1} t_{\ell+1} \right) = O(k),
\end{eqnarray*}
as argued before.

\medskip

\noindent \emph{Proof of $c_{new}(v,n) = \Omega(1/k)$}: The lower bound is straightforward. Clearly, $B_{L+1}$ is contained in the sphere of influence of vertex $v$ not only at time $n$ but, in fact, at any point of the process. It follows from Theorem~\ref{thm:X_t} that the number of vertices of in-degree 1 that lie in $B_{L+1}$ is $\Theta(b_{L+1} n) = \Theta(k)$. Moreover, their in-neighbours are also contained in the sphere of influence of $v$ and, with the desired probability, say, half of them are born after time $T_v$. In order to avoid complications with events not being independent, we can select a family of $\Theta(k)$ directed edges such that no endpoint belongs to more than one edge. Now, each of these selected edges have both endpoints in the in-neighbourhood of $v$ with probability $p^2$, independently on the other edges. Hence, the expected number of edges in $|E_{new}(N^-(v,n))|$ is $\Omega(1/k)$ and the conclusion follows easily from the Chernoff bound. 
\end{proof}

\subsection{Proof of Theorem~\ref{thm:negative}} \label{sec:negative}

We move immediately to the proof. 

\begin{proof}[Proof of Theorem~\ref{thm:negative}]
Let $1 \le \alpha=\alpha(n) = n^{o(1)}$ and $2 \le \beta=\beta(n) =O(\log n)$ be any functions of $n$. We will tune these functions at the end of the proof for a specific value of $k$, depending on the case (i), (ii), or (iii) we deal with. Pick any point $s$ in $S$ and consider two balls, $B_1$ and $B_2$, centered at $s$; the first one of volume $C_1$ and the second one of volume $C_2$, where
$$
C_1 = \frac {A_2}{10 n}, \quad \text{and} \quad C_2 = \frac {2(A_1+A_2) \beta}{ n / (2 \alpha) }. 
$$
Let $v$ be the first vertex that lands in $B_1$. We independently consider three phases. 

\bigskip

\noindent \textbf{Phase 1}: Up to time $T_1 = n/\alpha$ when $\deg^-(v,T_1) = \Theta(\beta)$. 

\smallskip

\noindent Consider the time interval between $n/(2\alpha)$ and $n/\alpha$. 
We are interested in the following event $D$: during the time interval under consideration, $\beta$ vertices land in $B_1$ but no vertex lands in $B_2 \setminus B_1$. Clearly, 
$$
\Prob (D) = \binom{n/(2\alpha)}{\beta} C_1^{\beta} (1-C_2)^{n/(2 \alpha) - \beta} \ge \left( \frac {nC_1}{3 \alpha\beta} \right)^{\beta} \exp \left( - \frac {C_2 n}{\alpha} \right).
$$
Straightforward but important observations are that every vertex in $B_1$ is inside a ball around any other vertex in $B_1$ (balls have volumes at least $A_2/(n / \alpha) \ge A_2/n$,  deterministically); moreover, conditioning on $D$, during the whole time interval all balls around $\beta$ vertices in $B_1$ are contained in $B_2$ (balls have volumes at most $(A_1 \beta + A_2)/(n/(2\alpha))$). 

We condition on event $D$ and consider two scenarios that will be applied for two different ranges of $k$.

\textbf{Event $F_1$}: vertices in $B_1$ form a (directed) complete graph on $\beta$ vertices; in particular, $\deg^-(v, n/\alpha) = \beta-1$ and $c^-(v,n/\alpha) = 1$. It follows that
$$
\Prob (F_1 | D) = p^{\binom{\beta}{2}},
$$
and so
\begin{eqnarray*}
\Prob (D \wedge F_1) &\ge& \left( \frac {nC_1}{3 \alpha\beta} \right)^{\beta} \exp \left( - \frac {C_2 n}{\alpha} \right) p^{\binom{\beta}{2}} \\
&=& \exp \left( - \beta \log \left( 30 \alpha \beta / A_2 \right) - 4 (A_1 + A_2) \beta - \binom{\beta}{2} \log \left( 1/p \right) \right) \\
&\ge& \exp \left( - \beta \log \alpha - 2 \beta \log \log n - \beta^{2} \log \left( 1/p \right) \right) \\
&\ge& n^{-1/5-o(1)-1/5} \ge n^{-1/2},
\end{eqnarray*}
provided that
\begin{equation}\label{eq:conditions_for_beta-F1}
\max \left\{ \beta \log \alpha, \beta^2 \log \left( 1/p \right) \right\} \le \frac 15 \log n.
\end{equation}

\textbf{Event $F_2$}: the first $\beta p/8-1$ vertices that landed in $B_1$ right after $v$ connected to $v$ but the remaining $\beta(1-p/8)$ vertices did not do this; moreover, each of $\beta p/8-1$ neighbours of $v$ got connected to at least $\beta p /4$ other vertices. In particular, $\deg^-(v,n/\alpha) = \beta p/8-1$ and all neighbours $w$ of $v$ satisfy $\deg^-(w,n/\alpha) \ge \beta p / 4$. It follows that
\begin{eqnarray*}
\Prob (F_2 | D) &=& p^{\beta p/8-1} (1-p)^{\beta (1-p/8)} \prod_{i=1}^{\beta p/8} \Prob \Big( \bin (\beta - i, p) \ge \beta p / 4 \Big) \\
&\ge& [p(1-p)]^{\beta} \ \Prob \Big( \bin (\beta (1-p/8), p) \ge \beta p / 4 \Big)^{\beta p/8} \\
&\ge& \left[ \frac {p(1-p)}{2} \right]^{\beta},
\end{eqnarray*}
since $\E ( \bin (\beta (1-p/8), p) ) = \beta(1-p/8)p \ge \beta p / 2$. This time we get
\begin{multline*}
\Prob (D \wedge F_2) \ge \left( \frac {nC_1}{3 \alpha\beta} \right)^{\beta} \exp \left( - \frac {C_2 n}{\alpha} \right) \left[ \frac {p(1-p)}{2} \right]^{\beta} \\
= \exp \left( - \beta \log \left( 30 \alpha \beta / A_2 \right) - 4 (A_1 + A_2) \beta - \beta \log \left( \frac {2}{p(1-p)} \right) \right) \\
\ge \exp \left( - \beta \log \alpha - \beta \log \beta - O(\beta) \right) \\
\ge n^{-1/5-1/5-o(1)} \ge n^{-1/2},
\end{multline*}
provided that
\begin{equation}\label{eq:conditions_for_beta-F2}
\max \left\{ \beta \log \alpha, \beta \log \beta \right\} \le \frac 15 \log n.
\end{equation}

\bigskip

\noindent \textbf{Phase 2}: Between time $T_1 = n/\alpha$ and time $T_2$ when $\deg^-(v,T_2) \ge \omega \log n$ for some $\omega = \omega(n) \le \log \log n$ tending to infinity as $n \to \infty$. 

\smallskip

\noindent We assume that events $D$ and $F_2$ hold. Let $W$ be the set of the first $\beta p/8-1$ neighbours of $v$ considered in the previous phase. Using the same argument as in Lemma~\ref{thm:conc}, we are going to show that with probability at least $1/2$ for any $t$ in the time interval under consideration and any vertex $w \in W \cup \{v\}$,
$$
\deg^-(w,t) \sim \deg^-(w,n/\alpha) \left( \frac {t}{n/\alpha} \right)^{pA_1}.
$$
Let $\eps = 1 / (\omega \log \log n)$ and suppose that 
$$
\deg^-(v,T)=d \ge \beta p / 8 - 1. 
$$
Then, with `failing' probability $\exp(-\Omega(\eps^2 d))$, for some value of $t$, $T \le t  \le 2T$, 
$$
\left| \deg^-(v,t) - d \cdot \left( \frac tT \right)^{p A_1} \right| > \frac {5}{p A_1} \cdot \frac {t}{T} \ \eps.
$$
We will apply this bound for $T = 2^i n / \alpha$ for $0 \le i = O(\log \log n)$. Hence, the probability that we fail for some vertex (at some time $t$ between $T_1$ and $T_2$) is at most
\begin{multline*}
\frac {\beta p}{8} \ O(\log \log n) \ \exp(-\Omega(\eps^2 d)) \\ = O(\beta \log \log n) \exp \left(-\Omega \left( \frac {\beta}{(\omega \log \log n)^2} \right) \right) \le \frac 12,
\end{multline*}
provided that 
\begin{equation}\label{eq:conditions_for_beta-Phase2}
\beta \ge \omega^3 \ (\log \log n)^2 (\log \log \log n).
\end{equation}
The claim holds as the cumulative error term is
$$
(1 + O(\eps))^{O(\log \log n)} = 1 + O(\eps \log \log n) \sim 1.
$$

\bigskip

\noindent \textbf{Phase 3}: Between time $T_2$ and time $n$.

\smallskip

\noindent We assume that events $D$ and $F_2$ hold, and Phase 2 finished successfully (that is, concentration holds for all vertices in $W$). It follows immediately from Corollary~\ref{cor:tT} that with probability $1-o(n^{-1} \beta)$ for any $t$ in the time interval between $T_2$ and $n$, and any vertex $w \in W \cup \{v\}$,
$$
\deg^-(w,t) \sim \deg^-(w,n/\alpha) \left( \frac {t}{n/\alpha} \right)^{pA_1}.
$$

\bigskip

The conclusion is that with probability at least $n^{-1/2}/3$, for a given point $s$ in $S$, there exists vertex $v$ in $B_1$ that has $\Theta(\beta)$ in-neighbours in $B_1$. Moreover, between time $n/\alpha$ and $n$, the degree of these neighbours of $v$ are larger by a factor of at least $2+o(1)$ than the degree of $v$. It follows that in this time interval, the ball around $v$ is contained in all the balls of early neighbours of $v$. Conditioning on this event and assuming that, say, $\alpha \ge 2$, with probability at least $1-\beta \exp(-\Omega(\omega \log n)) \ge 1-n^{-1}$, each early neighbour has a positive fraction of neighbours of $v$ as its neighbours at time $n$. (Note that this time events are not independent but the failing probability is small enough for the union bound to be applied.) It follows that with probability at least $n^{-1/2}/4$, we have $c^-(v,n) = \Omega (\beta / k)$.

Finally, tessellate $S$ into $n^{1-o(1)}$ squares of volumes, say,
$$
C_3 = \frac {\omega^2 \log n}{ n / (2 \alpha) } = n^{-1+o(1)}, 
$$
as it is assumed that $\alpha = n^{o(1)}$, and take various $s$ to be the centers of the corresponding squares. Note that conditioning of all the phases to end up with success, balls of all vertices under consideration are contained in the square. Moreover, in order to decide if a given square is successful does not require to expose vertices outside of this square. Hence, the events associated with different squares are almost independent. Formally, one would need to use (in a straightforward way) the second moment method to show this claim. It follows that a.a.s.\ there is at least one square that is successful.

\bigskip

Now, we are ready to tune $\alpha$ and $\beta$ for a specific function $k$. 
For case (i), we take $\alpha = 1$ (that is, no phase 2 and 3) and $\beta = k$. It is straightforward to see that conditions~(\ref{eq:conditions_for_beta-F1}) are satisfied. 
For case (ii), we take 
$$
\beta = \frac {k}{5} \le \frac {\log n}{5 \log \log n} \quad \text{and} \quad \alpha = \left( \frac {k}{\beta} \right)^{1/(pA_1)} = 5^{1/(pA_1)} \ge 5.
$$
(This time, there is no phase 3.) Again, it is straightforward to see that conditions~(\ref{eq:conditions_for_beta-F2}) and~(\ref{eq:conditions_for_beta-Phase2}) are satisfied. 
For case (iii), we take 
\begin{eqnarray*}
\beta &=& \frac {pA_1}{5} \omega (\log \log n)^2 (\log \log \log n) \quad \text{and} \quad \\ \alpha &=& \left( \frac {k}{\beta} \right)^{1/(pA_1)} \le k^{1/(pA_1)} \le n^{\xi/(pA_1)}.
\end{eqnarray*}
(Clearly, $\alpha \gg 1$.) As usual, it is straightforward to see that conditions~(\ref{eq:conditions_for_beta-F2}) and~(\ref{eq:conditions_for_beta-Phase2}) are satisfied, and the proof is finished. 
\end{proof}

\subsection{Proof of Theorem~\ref{thm:average}} \label{sec:average}

We move immediately to the proof. 
\begin{proof}[Proof of Theorem~\ref{thm:average}]
Let $\omega = \omega(n) = \log^{o(1)} n$ be any function tending to infinity as $n \to \infty$ (arbitrarily slowly). First, note that a.a.s.
$$
|X_k| = \sum_{i=(1-\delta)k}^{(1+\delta)k}  \Theta( i^{-1-1/(pA_1)} n ) =  \Theta_{\delta}(n k^{-1/(pA_1)}),
$$
as the degree distribution of $G_n$ follows power law with exponent $1+1/(pA_1)$~\cite{spa1}. Let 
$$
T = T(n) := n \left( \frac {2 \omega \log n}{k} \right)^{1/(pA_1)}.
$$
Note that $T \ge n^{\eps/(pA_1)}$, as $k \le n^{pA_1-\eps}$. It follows from Theorem~\ref{degconc} that a.a.s., for each $v \in X_k$, 
$$
(1+o(1)) (1-\delta) (2 \omega \log n) \le \deg^-(v,T) \le (1+o(1)) (1+\delta) (2 \omega \log n).
$$
In particular, for $n$ large enough,
\begin{equation}\label{eq:deg}
\deg^-(v,T) > \omega \log n 
\end{equation}
(as $\delta < 1/2$) and so all old neighbours of $v$ are born before time $T$. 

\bigskip

We start from part (i). As we aim for the statement that holds for almost all vertices in $X_k$, we may concentrate on any vertex $v \in X_k$ that is born after time $nk^{-1/(pA_1)}/\omega = o(nk^{-1/(pA_1)})$ and simply ignore the remaining ones (as the number of them is negligible comparing to $|X_k|$). Since each in-neighbour $v_u$ of $v$ is also born after time $nk^{-1/(pA_1)}/\omega$, we can use Corollary~\ref{cor:upper} to be able to assume that for any $u \le t \le n$, 
$$
\deg(v_u,t) \le \omega \log n \left( \frac {t}{u} \right)^{pA_1} \le \omega \log n \left( \frac {t}{nk^{-1/(pA_1)}/\omega} \right)^{pA_1}.
$$
As a result, for any $T \le t \le n$, 
\begin{multline*}
\frac {|S(v_u,t)|}{|S(v,t)|} \le (1+o(1)) \frac { \omega \log n \left( \frac {t}{nk^{-1/(pA_1)}/\omega} \right)^{pA_1} }{\deg^-(v,T) (t/T)^{pA_1}} \\ \le (1+o(1)) \left( \frac {T}{nk^{-1/(pA_1)}/\omega} \right)^{pA_1} 
\sim 2 \omega^{pA_1 + 1} \log n \le \omega^{2} \log n.
\end{multline*}
(Here we used Theorem~\ref{degconc} and~\eqref{eq:deg}.) Moreover, we may ignore all vertices that have too many vertices that are too close to them at time $T$. Formally, we ignore all vertices $v$ that have at least $C = \lceil 8/(\eps p A_1) \rceil > e$ vertices in the ball of volume $B=1/(T \log^{\eps/2} n)$ around $v$ at time $T$. Indeed, suppose that $T$ points are placed independently and uniformly at random in $S$ (without generating the graph). The probability that a given point $v$ has too many points around is at most
$$
\binom{T}{C} B^C \le \left( \frac {eTB}{C} \right)^C \le (TB)^C = \log^{-\eps C/2} n = \log^{-4/(pA_1)} n. 
$$
Since the expected number of such points is at most 
\begin{eqnarray*}
T \log^{-4/(pA_1)} n &=& n k^{-1/(pA_1)} (2 \omega \log n)^{1/(pA_1)} \log^{-4/(pA_1)} n \\ 
&\le& n k^{-1/(pA_1)} \log^{-2/(pA_1)} n,
\end{eqnarray*}
it follows from Markov's inequality that a.a.s.\ there are at most 
$$n k^{-1/(pA_1)} \log^{-1/(pA_1)} n = o(|X_k|)$$ 
of them, as claimed. (In fact, $n k^{-1/(pA_1)} \log^{-1/(pA_1)} n = O(|X_k| / (\omega^2 \log n))$, which will be needed for part (ii).)

Our goal is to show that $c_{old}(v,n) = O(1/k)$. Since there are at most $C = O(1)$ close in-neighbours of $v$, their contribution to $c_{old}(v,n)$ is only $O(1/k)$ and so we need to concentrate on far in-neighbours of $v$. Let 
$$
\hat{T} := T \log^{(2+\eps)/(1-pA_1)} n,
$$
and note that
$$
\hat{T} = n \left( \frac {2 \omega \log n}{k} \right)^{1/(pA_1)} \log^{(2+\eps)/(1-pA_1)} n = o(n),
$$
assuming that $k \ge \omega^2 \log^{1+(2+\eps)pA_1/(1-pA_1)} n$, which we may by taking $\omega$ small enough. 
Let $u$ be any (far) in-neighbour of $v$ that is outside of the ball of volume $B$ around $v$ at time $T$. Note that 
\begin{eqnarray*}
|S(u,\hat{T})| &\le& (\omega^2 \log n) |S(v,\hat{T})| \\
&\sim& (\omega^2 \log n) (A_1 \deg^-(v,\hat{T}))/\hat{T} \\
&\le& (\omega^2 \log n) (4 A_1 \omega \log n) (\hat{T}/T)^{pA_1}/\hat{T} \\
&\le& (\omega^4 \log^2 n) \hat{T}^{pA_1-1} T^{-pA_1} \\
&=& (\omega^4 \log^2 n) (\log^{-(2+\eps)} n) / T \\
&=& 1/(T \log^{-\eps + o(1)} n) = o(B)
\end{eqnarray*}
and so also $|S(v,\hat{T})| = o(B)$, which implies that at time $\hat{T}$ spheres of influence of $u$ and $v$ are disjoint and will continue to shrink. As a result, the number of common neighbours of $v$ and $u$ is at most
$$
\deg^-(v,\hat{T}) = k (\hat{T}/n)^{pA_1} = o(k),
$$
and so the number of common neighbours of $v$ and its far neighbours is negligible. Part (i) holds.

\bigskip

The proof of part (ii) is almost the same so we only point out small adjustments that need to be implemented. It follows from Corollary~\ref{cor:main} that we may assume that $c^-(v,n) = O(\omega \log n / k)$ for any vertex $v \in X_k$. Hence, it is enough to show that all but at most $O(|X_k| / (\omega^2 \log n)) = O(n k^{-1/(pA_1)} / (\omega^2 \log n))$ vertices in $X_k$ have $c^-(v,n) = O(1/k)$. This time we can only ignore vertices born before time $nk^{-1/(pA_1)}/(\omega^2 \log n)$ which gives slightly weaker bound for the ratio of the volumes of influence of a neighbour of $v$ and $v$ itself:
$$
\frac {|S(u,t)|}{|S(v,t)|} \le \omega^3 \log^{1+pA_1} n.
$$
As a result, we need to define $\bar{T}$ as a counterpart of $\hat{T}$ as follows:
$$
\bar{T} := T \log^{(2+pA_1+\eps)/(1-pA_1)} n,
$$
and note that $\bar{T} = o(n)$, assuming the stronger lower bound for $k$. The rest of the proof is not affected.
\end{proof}

\section{Conclusion and future directions}\label{sec:conclusion}

In this paper, we analyzed the clustering properties of the SPA model. Namely, we proved that the average local clustering coefficient $C(d)$ asymptotically behaves as $d^{-1}$. Moreover, we were able to prove that not only the average but the individual local clustering coefficient of a vertex of degree $d$ behaves as $1/d$ if $d$ is large enough.

The behaviour $C(d) \propto d^{-1}$ is often observed in real-world networks. However, in some cases $d^{-\varphi}$ with $\varphi \neq 1$ is claimed. In random graph models of complex networks, it is usually the case that $C(d) \propto d^{-1}$. For example, such decay is proved in~\cite{krot2017local} for a variety of networks based on preferential attachment.  

As a future work, it would be interesting to obtain a precise constant $C$ in the expression $C(d) \sim C d^{-1}$. We believe that the value $C$ is defined by the parameters $p$ and $A_1$ of the model, similarly to what was previously observed for the general class of preferential attachment models~\cite{krot2017local}. 
This would provide an insight into percolation properties of the network. 
Namely, it was shown in~\cite{serrano2006clustering2} that  percolation properties of a network are defined by the type
(weak or strong) of its connectivity and the type of connectivity is defined by the constant $C$ (whether it is greater or smaller than one). 
Another interesting direction for future research is to analyze the nature of usually obtained value $\varphi = 1$ and try to modify the SPA model in order to make it more flexible and allowing to generate graphs with $\varphi \neq 1$. Similar results with $\varphi \neq 1$ were recently obtained for affiliation networks~\cite{bloznelis2017correlation}.

\section*{Funding}

This work was supported by Russian Foundation for Basic Research [18-31-00207];
and Natural Sciences and Engineering Research Council of Canada.

\bibliographystyle{plain}
\bibliography{SPA}

\end{document}